\documentclass[11pt,a4paper]{article}
\usepackage{graphicx} 
\usepackage{a4wide} 
\usepackage[utf8]{inputenc}
\usepackage{amsfonts}
\usepackage{color}
\usepackage{subcaption}
\usepackage{enumitem}
\usepackage{mathtools}
\usepackage[nottoc]{tocbibind}
\usepackage{appendix}
\usepackage{bbm}
\usepackage{ amsthm  }
\usepackage{xfrac}
\usepackage{aligned-overset}
\newcommand{\gep}{\epsilon}

\newcommand{\E}{\mathbb{E}}

\newcommand{\N}{\mathbb{N}}

\newcommand{\R}{\mathbb{R}}

\newcommand{\V}{\mathbb{V}}

\newcommand{\Z}{\mathbb{Z}}

\title{A nonlocal traffic flow model with stochastic velocity} 
\author{Timo Böhme\footnotemark[1], \; Simone Göttlich\footnotemark[1], \; Andreas Neuenkirch\footnotemark[1]}
\date{\today}
\begin{document}
\theoremstyle{plain}
\newtheorem{satz}{Theorem}[section] 
\newtheorem{lemma}[satz]{Lemma}
\newtheorem{proposition}[satz]{Proposition}
\newtheorem{corollary}[satz]{Corollary}
\theoremstyle{definition}
\newtheorem{definition}[satz]{Definition}
\newtheorem{example}[satz]{Example}
\theoremstyle{remark}
\newtheorem{bemerkung}[satz]{Bemerkung}
\theoremstyle{remark}
\newtheorem{remark}[satz]{Remark}
\theoremstyle{remark}
\newtheorem{assumption}[satz]{Assumption}

\maketitle

\footnotetext[1]{University of Mannheim, Department of Mathematics, B6, 68159 Mannheim, Germany (\{goettlich,neuenkirch\}@uni-mannheim.de, ti.bohme@gmail.com).}

\begin{abstract}
\noindent
In this paper, we investigate a nonlocal traffic flow model based on a scalar conservation law, where a stochastic velocity function is assumed. 
In addition to the modeling, theoretical properties of the stochastic nonlocal model are provided, also addressing the question of well-posedness. 
A detailed numerical analysis offers insights how the stochasticity affects the evolution of densities.
Finally, numerical examples illustrate the mean behavior of solutions and the influence of parameters for a large number of realizations.

\end{abstract}



\section{Introduction} 
Macroscopic traffic models that incorporate hyperbolic conservation laws have been intensively studied in the last decades. For a comprehensive overview we refer to~\cite{garavellohanpiccoli2016book,GaravelloPiccoliBook}. 
These models have undergone numerous iterations and given rise to various offspring models.
An especially novel and actively researched class are 
nonlocal traffic flow models, see e.g.~\cite{BlandinGoatin2016,bressan2020traffic,Chiarello2018,Coclite2024,ColomboGaravelloMercier2012,colombo2020local, Friedrich2018,Huang2022,keimer2018nonlocal} for different directions.
This type of model allows not only the inclusion of local traffic but also the integration of information from traffic on the visible road ahead into the driving behavior.
From a practical perspective, the motivation of nonlocal traffic models might be that self-driving vehicles do not only rely on their own sensory data but also include remotely obtained data on the traffic situation down the whole road. Hence, in order to save resources and possibly travel time, a smart self-driving vehicle will reduce its speed anticipatory, even if congestions are further away.

However, it is well known that traffic as a whole is subject to stochastic influences and two major interpretations of such influence exist.
First, this influence can be truly stochastic,
as for example a human driver can only estimate the velocities of the other vehicles ahead or a self-driving vehicle may be prone to measurement errors.
Second, this can be understood as an imperfection to our knowledge, regarding the true behavior of the vehicles
for a perceived traffic situation.
In this view, vehicles behave deterministically, but our lack of knowledge regarding hidden variables prevents us from accurately predicting their behavior.
As empirical data indicates that deterministic relations fail especially for crowded roads, 
many stochastic extensions to local traffic models have been proposed, see e.g.~\cite{Jabari2012,Li2012,SopasakisKatsoulakis2006}.
In a broader sense, those models belong to the class of stochastic conservation laws with random fluxes, see e.g. \cite{Badwaik2021,Garnier2012,Mishra2012,Risebro2015}. 
Inspired by the approach proposed in~\cite{Li2012}, we aim to investigate how stochasticity may be treated within the framework of nonlocal traffic models that have been only considered deterministically so far. 
The stochastic model in~\cite{Li2012} assumes the fundamental diagram to be a random function, with an error term affecting the maximum velocity.
While the multiplicative influence of the random variable ensures smoothness of the flux, it is slightly more restrictive than the additive approach we will be proposing. 
However, this approach poses the implicit benefit that a noise term with a high enough lower bound will already lead to non-negative velocities at all times. 
Further, this implicitly enables the model to  behave on average similarly to its deterministic basis. 

For our investigations, we focus on a scalar nonlocal model depending on the downstream traffic velocity as originally presented in~\cite{Friedrich2018}. 
Our stochastic extension includes a perturbed velocity function, introducing uncertainty into the flux function.
The overall aim of this paper is to introduce a stochastic nonlocal traffic model and additionally examine its mathematical properties, thereby contributing to the active research of nonlocal models and stochastic traffic flow in general. 
The work program is as follows: 
In Section~\ref{sec:NV} we revisit the traffic model with nonlocal velocity from~\cite{Friedrich2018} and address the most important assumptions. 
Following this, in Section~\ref{sec:snv} we construct a theoretically and practically well suited error term such that the error does not overwhelmingly affect the model,
but remains close to being realistic for modeling any noise.
With such an error term established, we proceed to propose our stochastic model and conduct an analysis of random velocities from a theoretical point of view. 
Throughout these proofs, our goal is to demonstrate how the incorporation of stochastic elements and the thoughtful construction of the error term change the deterministic approaches.
Having established and implemented suitable numerical schemes in Section~\ref{sec:num schemes nonloc_stoch}, we continue to rigorously prove the existence and uniqueness of solutions to the underlying model,
as well as other central properties, as for example the validity of a maximum principle in Section~\ref{sec:prop_exist}.
The model's behavior is examined numerically in Section~\ref{sec:num_results}, 
demonstrating its reaction to different traffic situations and the impact of parameter choices on the solution, both averaged over a large number of realizations.

\section{Nonlocal velocity model (NV)}\label{sec:NV}

In this section, we briefly recall the original deterministic nonlocal velocity model introduced in \cite{Friedrich2018}.
While local models like the Lighthill-Whitham-Richards model~\cite{LighthillWhitham,Richards} consider a flux $f(\rho)=\rho v(\rho)$, dependent only on $\rho=\rho(t,x)$ for $t\in(0,T]$ and $x \in \R$, nonlocal models consider a weighted mean of the downstream traffic, resulting in a pathwise dependent flux. 
Model-wise, one can understand this behavior as a limit of a nonlocal Follow-the-Leader model, where a vehicle adapts its velocity dependent on a fixed amount of multiple leaders and their respective distances or their corresponding velocities \cite{Chiarello2020}.
This idea was originally developed based upon the downstream density \cite{BlandinGoatin2016}. However, we focus on the velocity function and thus the downstream velocity. 
Besides important results in terms of well-posedness and the inheritance of the maximum principle, the NV model shows a more natural behavior in the dispersion of congestions~\cite{Friedrich2021_diss}. 
Moreover, it is capable of modeling whole networks \cite{Friedrich2021}, making it a candidate for traffic flow optimization. As such, this model will become our main access point to stochastic nonlocal models. 

In the NV model a vehicle positioned at $x_0 \in \R$ evaluates the velocity over the interval $[x_0,x_0+\eta],$ where $\eta>0$ is the nonlocal range, i.e. the distance a driver or a self-driving vehicle is capable to assess. 
The dynamics can be described by the conservation law
\begin{align}\label{eq:NV}\tag{NV}
	\partial_t\rho+\partial_x(\rho(W_\eta * v(\rho)))=0,
\end{align}
where
\begin{align*}
	(W_\eta * v(\rho))(t,x):=\int_{x}^{x+\eta}  W_{\eta}(y-x) v(\rho(t,y)) \,dy, \quad \eta>0,
\end{align*}
and the Cauchy problem is equipped with initial conditions of the form
\begin{align}\label{eq:CP_BV}
	\rho(0,x)=\rho_0(x) \in (L^1 \cap \text{BV})(\R;[0,\rho^{max}])
\end{align}
for $\rho^{max}>0$ given. 
Here, $W_\eta$ is a kernel function denoting how the velocities are weighted with respect to the distance of an evaluating vehicle. 
We need to make the following assumptions on $W_\eta$ to obtain meaningful and stable solutions.

\begin{remark}\label{rem:ass_W} Given $\eta>0$, we assume that
$$ W_\eta \in C^1([0,\eta];\R^+) \;\;  \text{with} \;\;  W_\eta' \leq 0,\quad \int_{0}^{\eta} W_\eta(x) \,dx=W_0, \quad \lim_{\eta \to \infty} W_\eta(0)=0.$$
More insights on how to interpret these assumptions and counterexamples proving their necessity can be found in \cite{Chiarello2018} and \cite{Friedrich2021_diss}.
We remark that for the numerical experiments presented in this paper, we restrict to a concave kernel $W^{conc.}_\eta (x) = 3 (\eta^2-x^2) \frac{1}{2\eta^3}$, assigning higher weights to the velocities directly ahead of each vehicle compared to distant ones.
\end{remark}

Regarding the velocity function, the following assumptions, comparable to the ones of the LWR model, need to be made.
\begin{remark}\label{rem:ass_v_ND} Given $\rho^{max}>0$, we assume for the velocity $v(\rho):$
$$v \in C^2([0,\rho^{max}];\R^+) \;\; \text{with} \;\; v'\leq 0, \quad v(0)=v^{max}>0.$$
	We do not need to assume $v(\rho^{max})=0$ and it is apparent that $v(\rho) \leq v^{max}$ holds.
 \end{remark}

A major section of \cite{Friedrich2018}
and the cornerstone for our central Theorem \ref{thm:ex_uniq_sNV_extended} lies in establishing the existence and uniqueness of weak entropy solutions for the Cauchy Problem.
Due to the flux being pathwise dependent on the spatial domain some nonlocal theory has to be established to firstly define such solutions.
Referring to \cite[Def. 1]{BlandinGoatin2016} for a generalized definition of nonlocal weak solutions, we note that for entropy solutions in the sense of Kru\v{z}kov \cite{Kruzkov1970}, 
attention must also be paid to the partial derivatives
of $f(t,x,\rho):=\rho\bigl(W_\eta * v(\rho)\bigr)(t,x)$ with respect to $x$. As a special case of \cite[Def. 2]{BlandinGoatin2016} and thus as in \cite[Def. 2.11]{Friedrich2021_diss}, we define them as follows.

\begin{definition}[Nonlocal weak entropy solution]\label{def:nonlocal_weak_entropy_sol}\text{ }\\
	A function $\rho \in C([0,T];L^1(\R))$ with $\rho(t,\cdot) \in \text{BV}(\R;\R)$ is a weak entropy 
	solution to (\ref{eq:NV}) with (\ref{eq:CP_BV}), i.e. the Cauchy problem, if
	\begin{align*}
		\int_{0}^{T} & \int_{-\infty}^{\infty} |\rho-c|  \partial_t \phi +\text{sign}(\rho-c)(f(t,x,\rho)-f(t,x,c)) \partial_x \phi -\text{sign}(\rho-c) \partial_x f(t,x,c) \phi \,dx \,dt \\
		+& \int_{-\infty}^{\infty} |\rho_0(x)-c|\phi(0,x) \,dx \geq 0
	\end{align*}
	holds for all non-negative test-functions $\phi \in C_0^1([0,T] \times \R ; \R^+)$ and any constant $c\in \R$.
	This reduces to \cite[Def. 1]{BlandinGoatin2016} for bounded $\rho$ and special choices of $\phi$ and $c$ \cite[2.12]{Friedrich2021_diss}.
\end{definition}

\begin{remark}\label{rem:drop_entropy_cond}
The condition on the solution to (\ref{eq:NV}) being an \emph{entropy} one, to be unique can be dropped as discussed in \cite[2.2, 2.3]{Friedrich2021_diss}. The proof as done in \cite[2.3.2]{Friedrich2021_diss}, employs a fixed-point approach whilst showing that the characteristics do not cross. 
Yet, our approach maintains the entropy condition for the stochastic model to accommodate traditional proof techniques, while also indicating the potential for relaxing the entropy condition in our specific model as well.
\end{remark}

\section{Modeling of stochastic nonlocal velocity} \label{sec:snv}

We now derive our main model and present central results, while sketching how adding noise influences the comparable deterministic proofs. In particular, we will encounter technical challenges when adding noise. 
The proposed model  can be categorized as a nonlocal velocity model whose flux is subject to a random perturbation. 
From now on, we will work on a suitable underlying probability space $(\Omega,\mathcal{A},\mathbb{P})$. 

\subsection{Derivation}
We adopt a model-driven approach, adding a noise term to the considered quantity. 
In the case of the NV model, our approach is consequently, to add noise to the velocity
\begin{align*}
    \partial_t \rho + \partial_x \Bigl(\rho \bigl(W_\eta*(v(\rho)+\epsilon)\bigr)\Bigr) = 0,
\end{align*}
where $\epsilon=\{\epsilon(t), t \in [0,T]\}$ is a stochastic processes with suitable properties. 
We aim to add noise that induces meaningful and interpretable perturbations of the deterministic model.
As hinted, the reasons for adding noise include:
\begin{enumerate}
    \item Noise as an imperfection to the measured velocities: Either humans or autonomous vehicles might be prone to measurement errors. For this, it is reasonable to assume that the noise has mean zero and is bounded. 
    \item Noise as an uncertainty on the flux: Related to the idea in \cite{Li2012}, we might assume that the data is in fact correctly perceived, but the fundamental diagram, i.e. the flux-density relationship, is not given by a single relationship but is subject to random fluctuations. 
    This would imply that the same driver reacts differently and not deterministically to the same data. 
    Especially for high densities, empirical data indicate \cite[Fig. 1]{Wang2011} that a single flux-density relationship is not capable to represent the reality.  
    This leads to the same assumptions on $\epsilon$ as before, since an unbounded biased error would indicate a wrong underlying deterministic fundamental diagram.
\end{enumerate}

In comparison to the local stochastic model in~\cite{Li2012}, we also build our model on a stochastic velocity function.
However, besides using a nonlocal base-model, we do not restrict 
the stochastic influence to $v^{max}$, but imply a general disturbance on the perceived velocities themselves. 
While this approach allows us more flexibility in modelling, it necessitates careful consideration of the preservation of non-negativity of the fluxes.

\subsubsection*{Non-negativity of the fluxes}
Due to the above considerations we will assume that
\begin{align*}
    \sup_{t \in [0,T]} \sup_{\omega \in \Omega} |\gep(t, \omega)| \leq \tau 
\end{align*}
for some $\tau  < v^{max}$
with $v^{max}$ as in Remark \ref{rem:ass_v_ND} as well as
$$ \gep(0)\equiv0, \qquad \mathbb{E}[\gep(t)]=0, \quad t \in (0,T]. $$

Now, for $\rho$ large enough $v(\rho)$ becomes sufficiently small such that $v(\rho)-\tau < 0$ 
and therefore $$\mathbb{P} \big(v(\rho(t,x))+\epsilon(t) < 0\big)>0.$$ In order to obtain meaningful velocities we
have to employ a limiter in the form of
\begin{align}\label{eq:v_e}
    v_\epsilon(\rho,t)=v_{\epsilon(t)}(\rho(t,x),t):=\max \{0,v(\rho(t,x))+\epsilon(t)   \},
\end{align}
which  consequently perturbs the mean dynamics of the system for high densities, i.e. low velocities.
Obviously for low densities, i.e. high velocities, $v_\gep(\rho,t) \leq v^{max}$ does \emph{not} hold in general. 
However, we have $0 \leq v_\gep(\rho,t) \leq  v^{max} + \tau \leq 2  v^{max}$ which is still 
consistent with the assumption of a deterministic maximal velocity. 

\subsubsection*{Modeling of the noise}
Additionally to the assumptions (a) mean zero and (b) uniformly bounded (by $v^{max}$)
   on our noise,  one would ideally like to add the properties
  \begin{enumerate}  [noitemsep]
    \item[(c)] independent and identically distributed (i.i.d.) in time,
    \item[(d)]  optional dependence of the variance on $\eta$. 
\end{enumerate}

A first choice for $\epsilon$ would be to use a uniformly distributed white noise, that is to assume that
$$ \epsilon(t) \sim \mathcal{U}(-\tau ,\tau ), \qquad t \in (0,T],$$
where $\mathcal{U}(-\tau ,\tau $) denotes the uniform distribution over $[-\tau ,\tau ]$, and
that
\begin{align}\label{noise-naive}
	\epsilon(t_1),\,\epsilon(t_2) \quad \textrm{are (stochastically) independent for all } \,\, t_1 \neq t_2 \in (0,T].
\end{align}
However, it is a classical result, that such a process does not have a measurable realization as a map from $(\Omega, \mathcal{A})$ to $(\mathbb{R}^{[0,T]}, \mathcal{B}(\mathbb{R}^{[0,T]}))$, and objects as \(\int_0^T 	\epsilon(s,\omega) \, ds\)
are consequently not well defined. See, e.g. \cite[4.2]{Korezlioglu1992} for an analogous result for Gaussian white noise; the reasoning for the Gaussian case can be easily adapted.
Hence, it becomes necessary to relax the third criterion. 
We opt for a discrete-time error process that adequately satisfies condition (c) to a certain degree, while preserving the meaningfulness of our results.
 
For a fixed hyper-parameter $\delta_R$ and with $t^k=k\delta_R$, $R_T = \left\lfloor T/ {\delta}_R \right\rfloor$
 we use
\begin{align}\label{eq:construction_eps} 
\epsilon(t)=\sum_{k=1}^{R_T}  \epsilon^{k} \chi_{[t^k, t^{k+1}) }(t), \qquad t \in [0,T],  
\end{align}
where  $\epsilon^{1}, \ldots, \epsilon^{R_T}$ are i.i.d. $\mathcal{U}(-\tau ,\tau )$-distributed.
Thus, our noise will not not change instantaneously in time, but at a fixed time-grid. Due to this discrete time construction, the map $ \epsilon: (\Omega, \mathcal{A}) \rightarrow (\mathbb{R}^{[0,T]}, \mathcal{B}(\mathbb{R}^{[0,T]}))$ is measurable and the realizations $t\mapsto  \epsilon(t,\omega)$ are in particular bounded, integrable and BV.
Clearly, we loose property (c) for all $t_1 \neq t_2$.
We  now have for any $t_1,t_2 \in [0,T]$ with $t_1 \in [t^{n_1},t^{n_1+1}),\,t_2 \in [t^{n_2},t^{n_2+1})$ that
	\begin{align}\label{eq:naive_epsilon-2}
		\epsilon(t_1),\,\epsilon(t_2) \quad \textrm{are (stochastically) independent  if} \quad  n_1\neq n_2, 
	\end{align}
	and $\epsilon(t_1)=\epsilon(t_2)$ if $n_1 = n_2.$
   In particular, we obtain independence of  $\epsilon(t_1)$ and $\epsilon(t_2)$ if $|t_1-t_2| > \delta_R$.
For our numerical calculations we will generally assume that $ \delta_R \leq  \Delta t $. Thus, we will always be in the independent case, i.e. in the case of equation \eqref{eq:naive_epsilon-2}.
A comparison between two numerical time meshes with grid size $\Delta t$ and our noise time mesh with $\delta_R$ is visualized in 
Figure \ref{fig:numeric_vs_analytic}. 

\begin{figure}[htb!]
    \centering
    \includegraphics[scale=0.655]{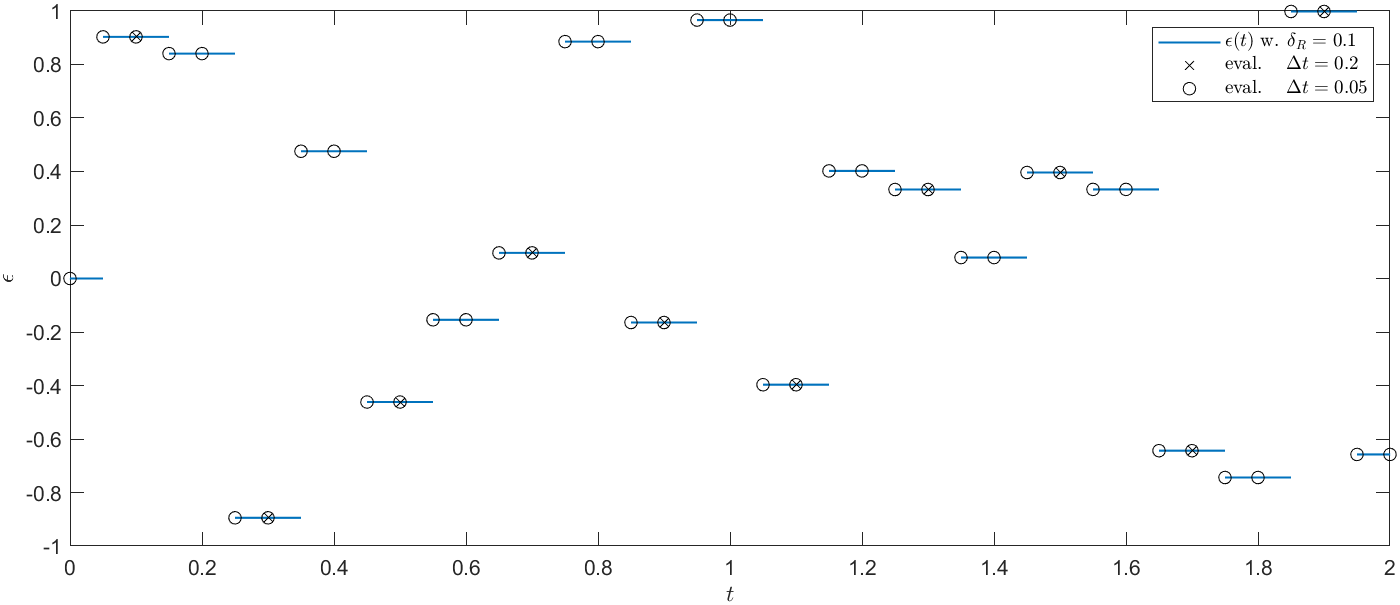}
    \caption{Realization of $\gep$ as constructed in (\ref{eq:construction_eps}) with $\tau=1$ and numerical evaluations.}
    \label{fig:numeric_vs_analytic}
\end{figure}

Later on, we also discuss another technical challenge of adding a noise term; namely the well-definedness of an entropy solution as a random variable.

\subsection{The stochastic nonlocal velocity model (sNV) }

The dynamics of the stochatic NV model are described by the following conservation law
\begin{align}\label{eq:sNV}\tag{sNV}
    \partial_t\rho+\partial_x \Bigl(\rho\bigl(W_\eta*v_\gep(\rho,t)\bigr)\Bigr)=0.
\end{align}
Here, the convolution is given by
\begin{align*}
    \Bigl(W_\eta * v_\gep(\rho,\cdot)\Bigr)(t,x):=\int_{x}^{x+\eta}  W_{\eta}(y-x) v_\gep(\rho(t,y),t) \,dy, \quad \eta>0.
\end{align*} 
The Cauchy problem is equipped with the initial conditions as in the deterministic model, cf. eq.~\eqref{eq:CP_BV}, 
with $\rho^{max}>0$ given and $v_\gep$ defined as in (\ref{eq:v_e}), 
where $\epsilon$ is given by \eqref{eq:construction_eps} with $\delta_R>0$ and $0 \leq \tau < v^{max}$.
In order to obtain a compact notation we abbreviate the convolution of random velocities similar as in (\ref{eq:NV}) and write 
\begin{align*}
    V_\gep(t,x):= \Bigl(W_\eta * v_\gep(\rho,\cdot)\Bigr)(t,x) \quad  \text{ and }  \quad f_\epsilon(t,x,\rho):= \rho V_\gep(t,x).
\end{align*}
For the the rest of this paper we fix without loss of generality $W_0=1$ 
and all of our simulations will be carried out for velocity functions satisfying $v^{max}=\rho^{max}=1$.\\
In addition to useful properties, we prove a central existence and uniqueness result to (\ref{eq:sNV}) over the course of Section \ref{sec:prop_exist}. 
We demonstrate, how the proofs rely upon the well-posedness of our stochastic velocity function along its associated error term.
Key factors include time integrability, the almost everywhere existing derivative of $v_\gep$ and its bounds alongside a presented numerical scheme with a suitable CFL condition.
Notably, the (pseudo) independence of the noise process is dispensable, enhancing the model's flexibility.

\subsection{Analysis of random velocities} \label{subsec:ana_v_e}
Prior to presenting our findings on the behavior and attributes of the sNV model, 
we discuss essential properties of the almost everywhere existing derivative of $v_\gep$ alongside an appropriate norm.
For now we assume a fixed time $t \in (0,T] $ and analyze the behavior of $v_\gep(\cdot,t)$, which  
we abbreviate 
as $v_\gep(\rho):=v_\gep(\rho,t)$.
The mapping
\begin{align*}
[0,\infty) \ni x \mapsto (x+c)^+ \in [0,\infty)
\end{align*}
is Lipschitz continuous with Lipschitz-constant one and satisfies
$$ (x+c)^+ - (y+c)^+ = \int_y^x \chi_{(0,\infty)}(z+c)dz, \qquad x,y \in \mathbb{R}.$$
 Consequently, the mapping
\begin{align*}
[0,\rho_{max}] \ni x \mapsto (v(x)+c)^+ \in [0,\infty)
\end{align*}
is  Lipschitz continuous with constant $\|v'\|_{\infty}$ and  of bounded total variation. 
Further, we have
\begin{align}\label{eq:MVP_leq}
(v(x)+c)^+ - (v(y)+c)^+ &= \int_y^x v'(z) \chi_{(0,\infty)}(v(z)+c)dz. 
\end{align} 
\begin{definition}\label{def:def_v'}
    Upon this consideration, we define 
\begin{align*}
    v_\gep'(\rho):=v'(\rho) \chi_{(0,\infty)}(v(\rho)+\gep).
\end{align*}
\end{definition}

While the actual derivative of $v_\gep$ coincides everywhere for $\gep\geq0$ with Definition \ref{def:def_v'}, such that we retrieve 
$v_{\gep} \in C^2([0,\rho^{max}];\R)$ again, it is not 
well defined at $\rho^*=v^{-1}(-\gep)$ for $\gep<0$.
For the latter case, the zero-set gap 
between the actual but undefined derivative and our definition
is negligible. 
This is due to the fact that given $v \in C^2([0,\rho^{max}];\R^+)$ and $\chi_{(0,\infty)} \geq 0$, we can invoke the mean value theorem for integrals to obtain for $x \neq y $:
\begin{align}\label{eq:MVP_eq}
(v(x)+c)^+ - (v(y)+c)^+ &= \int_y^x v'(z) \chi_{(0,\infty)}(v(z)+c)dz \nonumber \\ 
&= v'(\xi) \, \Pi(x,y) (x-y),
\end{align}
for some $\xi \in [x,y]$ with
\begin{align}\label{eq:Pi_x_y}
    \Pi(x,y):=\frac{1}{x-y}\cdot \int_y^x \chi_{(0,\infty)}(v(z)+c) \; \in (0,1].
\end{align}
Thus, mean value theorem arguments can be applied to $v_\gep$.
By the assumptions on the velocity function, $\sup_{\rho \in [0,\rho^{max}]}\left\lvert v_\gep'(\rho) \right\rvert$ is obtained on $\left\{\rho: v(\rho)+\gep >0 \right\}$ and
it follows immediately that
\begin{align}\label{eq:norm_approx}
       \sup_{\rho \in [0,\rho^{max}]}\left\lvert v_\gep'(\rho) \right\rvert
    \leq \sup_{\rho \in [0,\rho^{max}]}\left\lvert v'(\rho) \right\rvert 
    \quad \forall \gep \in [-\tau,\tau].
\end{align}

This analysis can now be transferred for variable $t \in (0,T]$ and $\omega \in \Omega$.
Here, Definition \ref{def:def_v'} translates in notation to
\begin{align*}
     v_\gep'(\rho(t,x),t)=v'(\rho(t,x)) \chi_{(0,\infty)}(v(\rho(t,x))+\gep(t)).
\end{align*}
As we are only interested in the spatial application of the mean value theorem and hence the spatial derivatives, we do not need an assessment of $\partial_t v_\gep(\rho,t)$, but must expand (\ref{eq:norm_approx}) to allow $v_\gep$ to be dependent on density and time. We start by defining an appropriate norm.
\begin{definition}\label{def:norm}
    For any $\omega \in \Omega$ we define the random variables
    \begin{align*}
        \left\lVert v_\gep \right\rVert (\omega) := \sup_{t \in [0,T]} \sup_{\rho \in [0,\rho^{max}]} \left\lvert v_\gep(\rho,t)(\omega)\right\rvert 
        \text{ and }
        \left\lVert v'_\gep \right\rVert (\omega) := \sup_{t \in [0,T]} \sup_{\rho \in [0,\rho^{max}]} \left\lvert v'_\gep(\rho,t)(\omega)\right\rvert, 
    \end{align*}
    while keeping the notation of $\left\lVert \cdot \right\rVert := {\left\lVert \cdot \right\rVert}_\infty$ for the one dimensional case, as before. 
\end{definition}
Now (\ref{eq:norm_approx}) adapts to the case of variable $t$ as the estimates hold independently of time, and we can add the notation with respect to $t$ and $\gep(t,\omega)$ to both sides. 
Thereby we can apply $\sup_{t \in [0,T]}$ to (\ref{eq:norm_approx}) and obtain 
\begin{align}\label{eq:Dv_e_norm_approx_final}
    \left\lVert v'_\gep \right\rVert (\omega) \leq \left\lVert v' \right\rVert 
    \quad \forall \omega \in \Omega. 
\end{align}
Hence, we have derived a deterministic bound on the spatial derivative of our random velocities. 
For completeness regarding $v_\gep$, it immediately holds
\begin{align}\label{eq:v_e_norm_approx_final}
    v^{max} = \left\lVert v \right\rVert \leq \left\lVert v_\gep \right\rVert (\omega) \leq v^{max}+\tau < 2 v^{max} \quad \forall \omega \in \Omega.
\end{align}

\subsection{Mean and variance of random velocities}\label{subsec:stoch_v_e}
Next, we briefly analyze the mean and variance of $v_\gep$ and $V_\gep$ to gain insight into the anticipated behavior of the sNV model itself.
We again fix any $t \in (0,T]$, an admissible density $\rho$ and abbreviate $v_\gep(\rho)(\omega):=v_\gep(\rho,t)(\omega)$.
Further, let the standard assumptions on $v$ and $W_\eta$ in Remarks \ref{rem:ass_W} and \ref{rem:ass_v_ND} hold.
For a fixed $t,$ the noise of (\ref{eq:construction_eps}) fulfills
\begin{align}\label{eq:exp_var_gep}
    \E[\gep(t)]=0,\quad \V(\gep(t))=\frac{1}{3}\tau^2  .
\end{align}
Straightforward calculations give:
\begin{lemma}\label{lem:exp_var_v}
    Let the Remarks \ref{rem:ass_W}, \ref{rem:ass_v_ND} hold, then
    \begin{align*}
        \E[v_\gep](\rho)&=\frac{1}{4\tau}\Bigl(\bigl(\tau+v(\rho)\bigr)^2-\max \bigl\{0,v(\rho)-\tau\bigr\}^2  \Bigr), \\
        \V(v_\gep)(\rho)&=\frac{1}{6\tau} \Bigl(\bigl(\tau+v(\rho)\bigr)^3-\max \bigl\{0,v(\rho)-\tau\bigr\}^3 \Bigr) - \E[v_\gep]^2(\rho). 
    \end{align*}
\end{lemma}
Since
$$ \max\{0,v(\rho(t,x)) +\epsilon(t)\} \geq  v(\rho(t,x)) +\epsilon(t),$$
we have
\begin{align}\label{eq:2.19}
	\E[v_\gep](\rho) \geq v(\rho) =  \E[v+\gep](\rho).
\end{align}
Moreover, since
$$ (\max\{0,v(\rho(t,x)) +\epsilon(t)\})^2 \leq  (v(\rho(t,x)) +\epsilon(t))^2,$$
it follows that
\begin{align}\label{eq:lower_var}
    \V(v_\gep)(\rho) \leq \V(\gep)=\V(v(\rho)+\gep).
\end{align}
Thus, adding the noise leads to an increase in the mean velocity, while the variance of the velocity is bounded by the variance of the noise.
Finally, note that for large velocities, i.e. $v(\rho)>\tau$, the limiter is not active and we have
$$    \E[v_\gep](\rho)  =  v(\rho)   , \qquad  \V(v_\gep)(\rho) = \V(\gep).$$

These observations are visualized in Figure \ref{fig:stoch_on_v_eps}, 
where the upper and lower bounds of $v_\gep(\rho)$, i.e. $\max\{0,v(\rho) \pm \tau\}$, are displayed by the black lines 
and the area where the limiter \emph{can} become active, i.e. to the left and right of $\rho^*=v^{-1}(\tau)$, by the vertical lines.\\

\begin{figure}[htb]
    \centering
    \begin{subfigure}{.5\textwidth}
        \centering
        \includegraphics[width=1\linewidth]{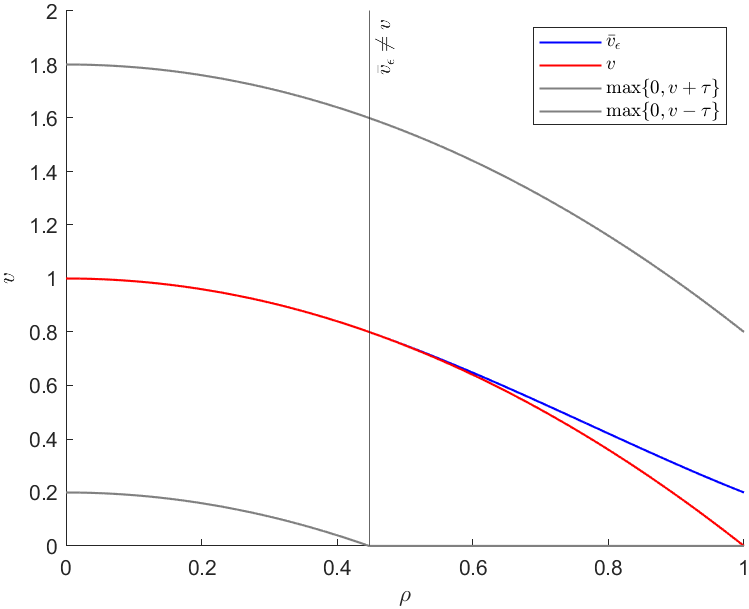}
      \end{subfigure}%
      \begin{subfigure}{.5\textwidth}
        \centering
        \includegraphics[width=1\linewidth]{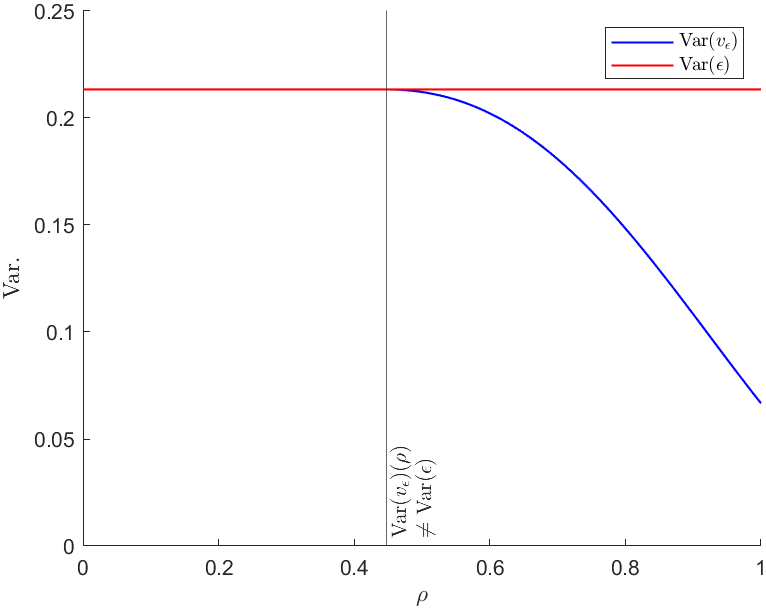}
      \end{subfigure}
    \caption{$v(\rho)=1-\rho^2$, the expectation ($\overline{v}_\gep$) of $v_\gep$ and its variance, given $\tau=0.8$.}
    \label{fig:stoch_on_v_eps}
\end{figure}
\vspace{-0.3cm}
Regarding the actual velocity of interest, $V_\gep$,
the behavior of $v_\gep$ translates but is additionally influenced by $W_\eta$ and $\eta$.
Due to $W_\eta,\, v_\gep \geq 0$, Fubini's Theorem, Lemma \ref{lem:exp_var_v} and the monotonicity of $W_\eta$, it holds
\begin{align}\label{eq:2.22}
    \E[V_\gep(t,x)](\rho) = \bigl(W_\eta*\E\left[v_\gep\right] \bigr) \bigl(t,x\bigr)(\rho) 
    \geq \bigl(W_\eta*v \bigr) \bigl(t,x\bigr)(\rho)
    =V(t,x)(\rho).  
\end{align}
Especially the deviation of $\E[V_\gep]$ from $V$ is not only obtained for $\rho \geq \rho^*$ but already for $\rho \geq \rho^*-\eta$. 
However, due to the nature of the kernel function with respect to $\eta$ and the proportionality of the velocities in the convolution, the strength of the deviation in the area $[\rho^*-\eta,\rho^*]$ can be rather weak.
Lastly we define for further reference
\begin{align*}
    \overline{v}_\gep(\rho):=\E[v_\gep](\rho) \quad \text{ and } \quad \overline{V}_\gep(t,x)(\rho):=\E[V_\gep(t,x)](\rho).
\end{align*}
    
\subsection{Characteristics} \label{subsec:chars}
Next, we initiate the examination of the model's behavior through its characteristics.
More precisely, we provide empirical evidence that the characteristics do not cross, show how the noise changes the behavior of (\ref{eq:NV}) and how the average behaviour may be captured. 
Here we combine the notation of \cite[1.1]{Holden2015} with \cite[Def. 2.3.2]{Friedrich2021_diss}.

\begin{definition}[Characteristics of (\ref{eq:sNV})]\text{ }\\
Let 
    \begin{align*}
    V_\gep[\rho](t,x):=(W_\eta * v_\gep(\rho,\cdot))(t,x)
    \end{align*}
and for any given $\omega \in \Omega$, let $\rho=\rho(\omega)$, be a weak solution to (\ref{eq:sNV}).  
Then the characteristics $X_{\rho,\gep}:(0,T) \times \R \times (0,T) \rightarrow \R$ are the solutions
to the integral equation
    \begin{align*}
        X_{\rho,\gep}[t_0,x_0](t):=x_0 + \int_{t_0}^{t} V_\gep[\rho] \Bigl(s, X_{\rho,\gep}[t_0,x_0](s) \Bigr) \, ds, \qquad t \in [t_0,T],
    \end{align*}
    with  $(t_0,x_0) \in [0,T] \times \R$.
Thus, they are the solution to the ODE
    \begin{align}\label{eq:char_ODE}
        \frac{d}{dt} X_{\rho,\gep}[t_0,x_0](t) &= V_\gep[\rho] (t, X_{\rho,\gep}[t_0,x_0](t)), \qquad t \in [t_0,T], \\
        X_{\rho,\gep}[t_0,x_0](t_0) &= x_0. \nonumber
    \end{align}

\end{definition}

The well-posedness of (\ref{eq:char_ODE}) is ensured by \cite[2.34]{Friedrich2021_diss}, as the proof directly translates due to the spatial differentiability of $V_\gep$ 
with slightly different bound $\mathcal{V}_{\gep}^\infty$, which can be found in the Appendix as equation (\ref{eq:V_cal}). \\

\vspace{-0.3cm}
In order to evaluate (\ref{eq:char_ODE}) numerically,
we draw one realization of $\gep$ and approximate (\ref{eq:sNV}) for $\Delta x= 10^{-2}$ with the scheme (\ref{eq:SCHEM_sNV}) and $\Delta t$ satisfying (\ref{eq:sCFL}).
Hence, we obtain the time and space dependent velocities $V_\gep[\rho]$ and densities $\rho$. 
Next, we choose any $(t_0,x_0) \in \{0\}\times\R$  and invoke explicit Euler time-marching (e.g. \cite[1.1]{Holden2015}) with the same time mesh as before. 
Doing this, we approximate the evaluation of $V_\gep$ in the spatial coordinate at $X_{\rho,\gep}[t_0,x_0](t)$ by the closest known evaluation on the grid defined by $\Delta x$.
For comparability, we choose the same data as in \cite[2.33]{Friedrich2021_diss}, i.e. the initial data $\rho_0$
from Example \ref{ex:standart_ex} with nonlocal range $\eta=0.1$. 

\begin{example}\label{ex:standart_ex}
We consider the initial data 
\begin{align*}
    \rho_0=
    \begin{cases}
	1, & \text{if } x\in\left[ \frac{1}{3},\frac{2}{3} \right] \\
	\frac{1}{3}, & \text{else}
    \end{cases}
\end{align*}
 and velocity function $v(\rho)=1-\rho$, convoluted with a concave kernel $W_\eta^{conc.}$.
\end{example}
The results are plotted in Figure \ref{fig:char_1-p}. 
To demonstrate the impact of the error term, we  implied low disturbances ($\tau=0.2$) on the left and high disturbances ($\tau=0.9$) in the right graphs.\\
\begin{figure}[htb]
    \centering
    \begin{subfigure}{.5\textwidth}
        \centering
        \includegraphics[width=1\linewidth]{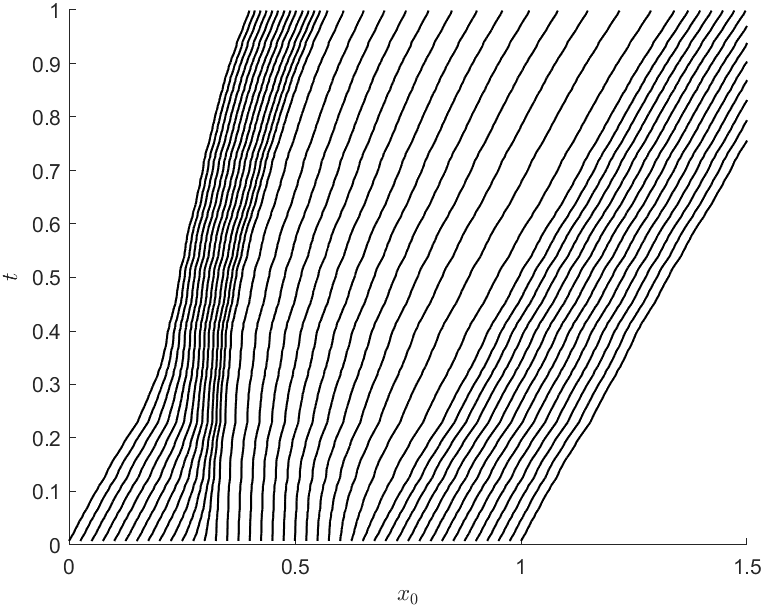}
      \end{subfigure}%
      \begin{subfigure}{.5\textwidth}
        \centering
        \includegraphics[width=1\linewidth]{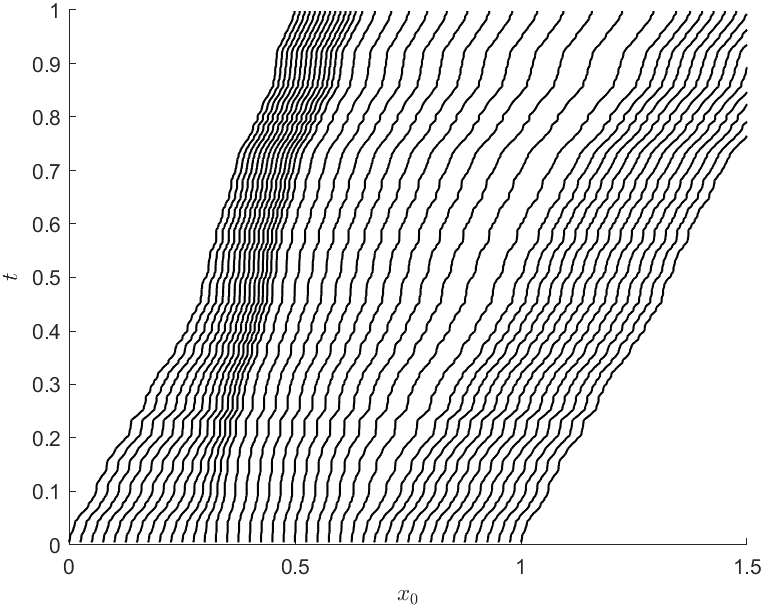}
      \end{subfigure}
    \caption{Characteristics of Example \ref{ex:standart_ex} to (\ref{eq:sNV}).} 
    \label{fig:char_1-p}
\end{figure}

As mentioned,
the characteristics do not cross, since at every time evaluation with stepsize  $\max\{\Delta t, \delta_R\}$ the slope of the characteristics changes everywhere on the spatial domain by the same degree.
The maximal and average slope-change is determined by the distribution of $\gep$, or rather $v_\gep$, and thereby dependent on
$\tau$. 
This does not hold in general, if the error term becomes spatially dependent.\\

We repeat our experiments on some slightly changed example, where the area of high 
initial density has been expanded for better visualization.
\begin{example}\label{ex:enalarged_ex}
    We consider the initial data
	\begin{align*}
		\rho_0=
		\begin{cases}
			1, & \text{if } x \in \left[ 0,1 \right] \\
			\frac{1}{3}, & \text{else}
		\end{cases}
	\end{align*}
    and velocity function $v(\rho)=1-\rho^2$, convoluted with a concave kernel $W_\eta^{conc.}$. 
\end{example}
Considering the analytical setup $\eta=0.2,\, \tau=0.8$, we then proceed as before to obtain
the characteristics. Instead of plotting only one realization of $\omega \in \Omega$,
we now simulate multiple ($N=30$) realizations of the characteristics and compare them to the deterministic ones
of (\ref{eq:NV}), in Figure \ref{fig:sNV_vs_NV_characteristcs}.\\

    \begin{figure}[htb!]
        \centering
        \includegraphics[scale=0.3538]{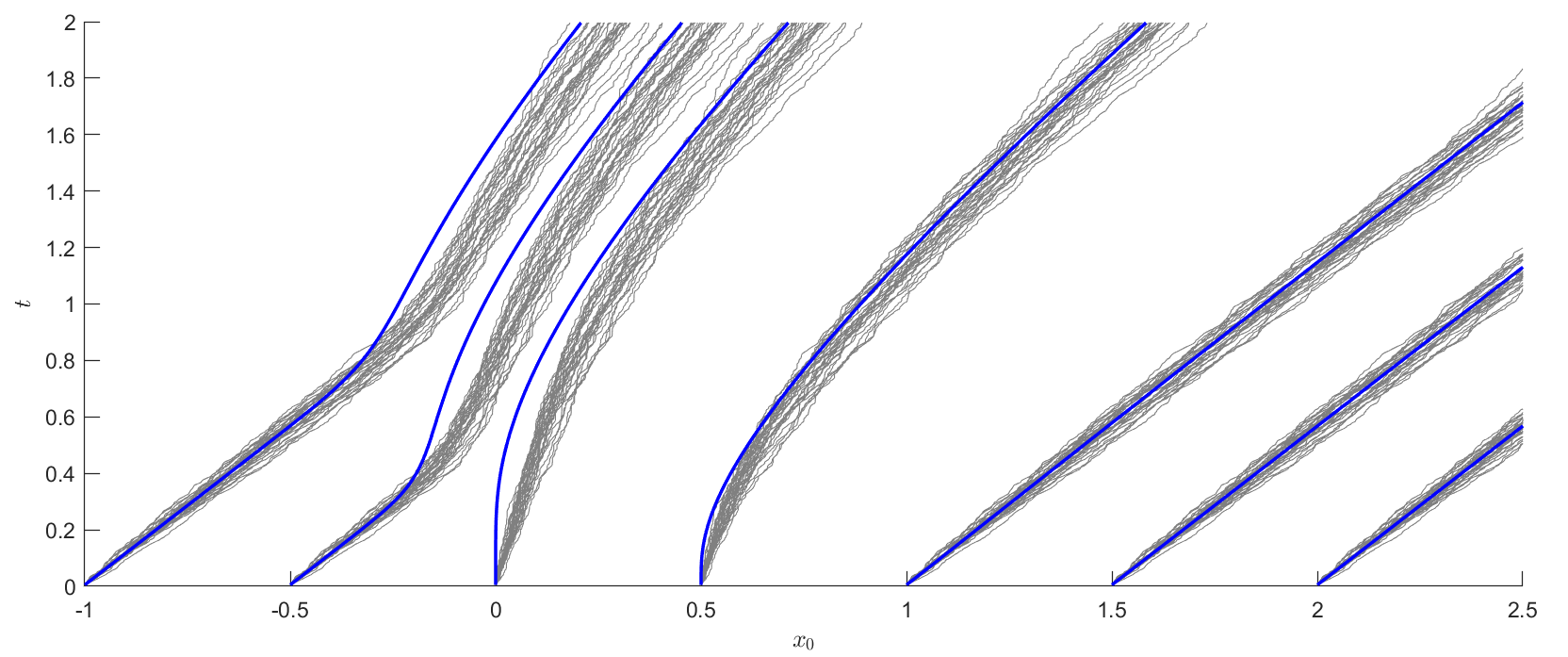}
        \caption{Characteristics of Example \ref{ex:enalarged_ex} to (\ref{eq:sNV}) in grey and
        (\ref{eq:NV}) in blue.}
        \label{fig:sNV_vs_NV_characteristcs}
    \end{figure}
    
Notice that vehicles, which always travel in an area of low downstream densities, i.e. right of $x_0=1$, behave on average like (\ref{eq:NV}).
However, vehicles entering an area of high downstream density behave differently.
For example, a vehicle placed at $x_0=-0.5$ moves initially 
in an area of low downstream density but then enters a congested area.
Hence, whilst being mean consistent at first, this changes as the evaluation ahead ($\eta=0.2$) picks up the congestion. 
Vehicles of (\ref{eq:sNV}) move faster in these areas, which is observable by the on average flatter characteristics and is backed up, by
our previous assessment of the expected velocities (\ref{eq:2.22}).\\

\begin{figure}[htb!]
    \centering
    \includegraphics[scale=0.3538]{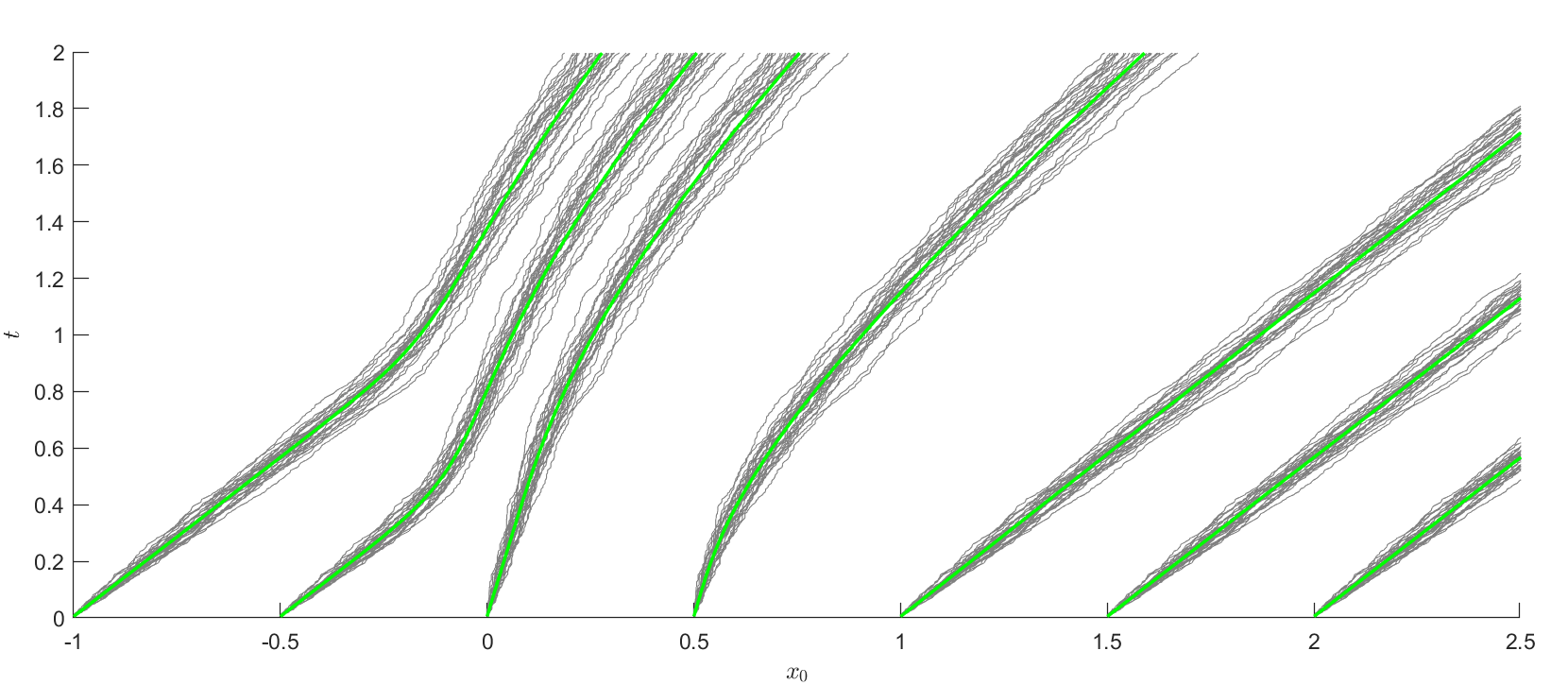}
    \caption{Characteristics of Example \ref{ex:enalarged_ex} to (\ref{eq:sNV}) in grey and
    (\ref{eq:NV}) using $\bar{v}_\gep$ in green.}
    \label{fig:sNV_vs_NV_characteristcs_expected}
\end{figure}
Next, we repeat the previous experiment in Figure \ref{fig:sNV_vs_NV_characteristcs_expected}, but compare the characteristics of $30$ (\ref{eq:sNV})-realizations with (\ref{eq:NV}) where we used the expected velocity function $\bar{v}_\gep$ instead of $v$ in the latter.
We observe that the resulting characteristics empirically align with the mean behavior of
(\ref{eq:sNV}). 
This observation, alongside further extensive Monte Carlo experiments conducted on the density level, leads to the conjecture that the mean density of (\ref{eq:sNV}) coincides, or can at least be adequately captured, by (\ref{eq:NV}) using $\bar{v}_\gep$. 
We intend to validate this hypothesis through additional numerical experiments and to-be-developed analytical methods in the future.

\section{Numerical scheme for the sNV model} \label{sec:num schemes nonloc_stoch}

We now enhance a deterministic nonlocal Godunov type scheme \cite[3.1]{Friedrich2018} to additionally incorporate the noise of the sNV model. Hence, the following coincides with  the deterministic scheme by formally setting $\gep \equiv 0$.

Assume an equidistant spatial grid, with cell centers $x_j$, cell interfaces $x_{j-1/2}$ and cell length $\Delta x= x_{j+1/2}-x_{j-1/2} \; \forall j \in \Z$. 
Further, we set the time mesh by $t^n=n\Delta t$ for $n=0,\dots,N_T$ with $N_T:=\left\lfloor T/\Delta t \right\rfloor$.
Let as usual $\rho_j^n:=\rho(t^n,x_j)$ and define the piecewise constant function 
\begin{align}\label{eq:mesh}
	\rho^{\Delta x}(t,x)=\rho_j^n \text{ for } (t,x)\in[t^n,t^{n+1})\times[x_{j-1/2},x_{j+1/2}).
\end{align} 
Then, the deterministic initial density $\rho_0$ is discretized by cell averages with respect to (\ref{eq:mesh}):
\begin{align*}
    \rho^0_j=\frac{1}{\Delta x} \int_{x_{j-1/2}}^{x_{j+1/2}} \rho_0(x)  \,dx, \quad j \in \Z. 
\end{align*}

In each time step the Riemann problems arising at the discontinuities between the numerical densities $\rho^n_j, \; j \in \Z$ are then solved exactly until the first shocks collide. Thus, the update of the cells densities is calculated as
\begin{align*}
    \rho_j^{n+1}=\rho_j^{n}+\frac{\Delta t}{\Delta x}\left(F^n_{j+1/2}(\rho_j^n)-F^n_{j-1/2}(\rho_{j-1}^n)\right),
\end{align*}
where the numerical flux $F^n_{j+1/2}$ is based on the solution to the Riemann problems at the cell interfaces and the actual flux.
To incorporate the error process, we use the technical simplification $\delta_R \leq \Delta t$ and draw a numerical evaluation of $\gep(t,\omega),\, t \in[0,T]$, by sampling $N_T-1$ observations according to (\ref{eq:construction_eps}):
\begin{align*}
    \gep^n:=\gep(t^n) \overset{\mathrm{iid}}{\sim} U(-\tau ,\tau ), \quad n=1,\dots,N_T
\end{align*}
and fix $\gep^0 = 0$.
We extend the discrete notation of $v_\gep$ to allow for the time dependency, introduced by $\gep^n$. 
Therefore, we define 
\begin{align}\label{eq:def_v^n}
    v_{\gep}^n(\rho_j^n)&:=v_\gep(\rho_j^n,t^n)=\max\{0,v(\rho_j^n)+\gep^n\},
\end{align}
dropping the index of $\gep^n$ with respect to $n$ on the left-hand side of the equation for a cleaner notation. 
This allows us to employ the numerical flux
\begin{align}\label{eq:num_flux}
	F_{j+1/2}^n(\rho_j^n)=\rho_j^n V_{\gep,j}^n, \;\; \text{with} \;\; V_{\gep,j}^n &= \sum_{k=0}^{N_\eta-1} \gamma_k v_\gep^n(\rho^n_{j+k+1}), \quad N_\eta=\left\lfloor \eta/\Delta x\right\rfloor,
\end{align}
given the kernel evaluation
\begin{align}
	\gamma_k &=\int_{k \Delta x}^{(k+1) \Delta x} W_\eta(x) \,dx, \quad k=0,\dots,N_\eta-1. \label{eq:gamma_k}
\end{align}

Thus, our time step update reads as
\begin{align}\label{eq:SCHEM_sNV}
	\rho_j^{n+1}=\rho_j^{n}-\lambda \left(\rho_j^n V_{\gep,j}^n-\rho_{j-1}^n V_{\gep,j-1}^n\right), \quad \lambda:=\frac{\Delta t}{\Delta x}.
\end{align}
Simultaneously as for the continuous case, the influence of stochastic transfers to $F^n_{j+1/2}$ and $\rho_j^n$.
Yet, notation with respect to $\gep$ of both shall be omitted.
 Due to the presumed accurate calculation of
  $\gamma_k \geq 0$ and our construction of $v_\gep$ we obtain
 \begin{align}\label{eq:V_gep_grt_0}
    0 \leq V_{\gep,j}^n \leq v^{max}+\tau \quad \forall j \in Z,\; \forall \omega \in \Omega,\; \forall n=0,\dots,N_T,
 \end{align}
 such that the Riemann problems are correctly solved by (\ref{eq:num_flux}). 
Next, we develop a fitting CFL condition and derive its deterministic bounds.
By definition of the norm $\lVert v_\gep \rVert$ (Def. \ref{def:norm}) for all $\omega \in \Omega$ it holds that
\begin{align*}
    {\lVert v_\gep^n \rVert}_{\infty}(\omega) = \sup_{\rho \in [0,\rho^{max}]} v_\gep(\rho,t^n)(\omega) 
    \leq \sup_{t \in [0,T]} \sup_{\rho \in [0,\rho^{max}]} v_\gep(\rho,t)(\omega)
    = \lVert v_\gep \rVert(\omega) 
    \overset{(\ref{eq:v_e_norm_approx_final})}{\leq} v^{max}+\tau
\end{align*}
and analogous
\begin{align*}
    {\lVert {v_{\gep}^n}' \rVert}_{\infty}(\omega) &\leq \lVert {v_\gep}' \rVert(\omega) 
    \overset{(\ref{eq:Dv_e_norm_approx_final})}{\leq}\lVert {v}' \rVert 
\end{align*}

Hence, proposing the following CFL condition for (\ref{eq:sNV})
\begin{align}\label{eq:sCFL} 
    \lambda \leq \frac{1}{\gamma_0 \lVert v' \rVert \rho^{max}+\lVert v_\gep \rVert(\omega)}, 
\end{align}
we obtain the following bounds
\begin{align*}
    \frac{1}{\gamma_0 \lVert v' \rVert \rho^{max}+v^{max} + \tau} 
    \leq \frac{1}{\gamma_0 \lVert v' \rVert \rho^{max}+\lVert v_\gep \rVert(\omega)} 
    \leq \frac{1}{\gamma_0 \lVert {v_{\gep}^n}' \rVert(\omega) \rho^{max}+\lVert v_\gep^n \rVert(\omega)}. \\
\end{align*}

\begin{remark}\label{rem:on_CFL}\leavevmode 
    \begin{itemize}
        \item The first inequality allows us to use a little more restrictive condition, without actually analyzing the
            occurring velocities to obtain a deterministic and easy to implement version of our CFL condition. 
            The second inequality provides initial information on the validity of this condition, as it effectively bounds
            the velocity and its derivative in every time step.
        
        \item Due to (\ref{eq:v_e_norm_approx_final}), equation (\ref{eq:sCFL}) is more restrictive than its deterministic counterpart
        \begin{align}\label{eq:CFL_NV}\tag{CFL}
	   \lambda \leq \frac{1}{\gamma_0 \left\lVert v' \right\rVert \rho^{max} + \left\lVert v \right\rVert },
        \end{align}
        whereas both coincide for $\tau=0$.
        Counterexamples that (\ref{eq:CFL_NV}) is not sufficient for (\ref{eq:sNV}) can easily be made.
    \end{itemize}
\end{remark}

\begin{remark}\label{num-mb}{
    By construction, the discretizations as given by the maps
    $$ \Omega \ni \omega \mapsto \rho_j^n(\omega) \in [0,\infty), \qquad   \Omega \ni \omega \mapsto V_j^n(\omega) \in  [0,\infty)$$
    are measurable from $(\Omega, \mathcal{A})$ to $([0,\infty),\mathcal{B}([0,\infty)))$ for all $n=0, \ldots, N_T$, $j \in \mathbb{Z}$, since they are constructed as measurable transformations of  the random variables $\epsilon(t^n)$, $n=0, \ldots, N_T$. 
    Thus, quantities as 
    $\mathbb{E} [\rho_j^n(\omega)]$ or quantiles of their distribution are well defined.}
\end{remark}
\section{Properties, existence and uniqueness }\label{sec:prop_exist} 
 Relying on the numerical scheme, we derive
 our central theorem regarding the existence and uniqueness of solutions to the Cauchy problem
 as well as necessary and helpful properties of the sNV model.
 To this end, we outline how the argumentation for the NV model transfers to
 the the sNV model, derive which modifications to the proofs are to be made and
 what any stochastic velocity function and its noise terms need to fulfill. 
 Since most of the proofs are rather technical, they are fully given in the appendix. 
 
\subsection{Properties of the sNV model}\label{subsec:properties}

\begin{lemma}[Discrete maximum principle]\label{lem:maxP}\text{ }\\
    Let the assumptions on $W_\eta$ and $v$ hold.
    Further, let the random variables $\rho_j^n$ be constructed by the scheme (\ref{eq:SCHEM_sNV}) with the CFL condition (\ref{eq:sCFL}).
    Then, the following holds
    \begin{align}\tag{\textasteriskcentered{}}\label{eq:max_p_*}
        \inf_{j \in \Z} \rho_j^0 \leq \rho_j^n \leq \sup_{j \in \Z} \rho_j^0 \quad \forall j \in \Z, \, n \in \N.
    \end{align}
\end{lemma}
\begin{proof}
    As in \cite[3.3]{Friedrich2018}, we prove the claim by induction in time, whilst applying our findings on 
    $v_\gep$ and $v_\gep'$ from Section \ref{subsec:ana_v_e}. 
     The adaptation of the proof for the NV model is threefold. 
    First, the discretization $v_\gep^n$ must still be monotone decreasing for any realization of $\gep$, which is given as we have seen before. 
    Second, our analysis of $v_\gep'$ and thus $(v_\gep^n)'$ as in (\ref{eq:MVP_leq})  allows us to invoke a mean value theorem argumentation.
    Last, the adapted CFL condition 
    provides the final inequality. 
    Further note, that the bounds are in fact independent of the realization of the random velocities.   
    The full proof is provided in Section \ref{lem:maxP_app}.
\end{proof}

From the maximum principle we obtain the positivity of solutions as $0 \leq \rho_0(x)$ implies \linebreak
 $0\leq \rho_j^n(\omega)$ for
all realizations of the random variable. Hence, one can show as for the NV model that 
the numerical scheme preserves the $L^1(\R)$ norm.

\begin{corollary}\label{cor:L1_conservation}
    Let $\rho_0(x)$ as in (\ref{eq:CP_BV}) and the usual assumptions
    on $W_\eta$ and $v$ hold. If $\rho_j^n$ is constructed by the scheme (\ref{eq:SCHEM_sNV})
    with the CFL condition (\ref{eq:sCFL}), then
    \begin{align*}
        \Delta x \sum_{j \in \Z} |\rho_j^n (\omega)| = {\left\lVert \rho_0\right\rVert}_{L^1(\R)}, \quad \forall n \in \{0,\dots,N_T\},\; \forall \omega \in \Omega.
    \end{align*}
\end{corollary}
\begin{proof}
    Due to the fact that we only perturb the velocities but not the quantities, the conservation of mass still applies
    in every time step. As we have further proven the maximum principle,
    the claim can now be easily shown by induction, as in \cite[2.19]{Friedrich2021_diss}, where only changes in notation apply.    
\end{proof}

Next, we show that the densities obtained by the scheme are of bounded total variation. 
To do this, we backtrack from $\rho_j^n$ to the initial densities 
$\rho_j^0$, which are of bounded variation by assumption, mainly combining the proofs \cite[2.20, 2.22]{Friedrich2021_diss} and 
\cite[3.4, 3.5]{Friedrich2018}. 
Note, that the naive construction of $\gep$ as done in (\ref{noise-naive})
does not only raise issues with respect to the measurability of the realizations, but also can not lead to bounded variation of the densities, as $\gep \notin \text{BV}((0,T])$ in that case.
In the following we derive $\gep$-dependent as well as additional $\gep$-independent bounds, to decouple the influence of
the stochastic velocities from the BV estimates. 

\begin{lemma}[BV estimate in space]\label{lem:BV_space}\text{ }\\
    Let $\rho_0(x)$ as in (\ref{eq:CP_BV}) and $\gep(t,\omega)$ as in (\ref{eq:construction_eps}). 
    Further, let the usual assumptions on $W_\eta$ and $v$ hold.
    If $\rho^{\Delta x}$, resp. $\rho_j^n$, is constructed by (\ref{eq:SCHEM_sNV})
    with the CFL condition (\ref{eq:sCFL}), then for all $T>0$ we have the $\gep$-dependent discrete space BV
    estimate:
    \begin{align*}
        \text{TV}(\rho^{\Delta x}(T,\cdot);\R)
        \leq \exp \bigl(T C \bigr) \text{TV}(\rho_0;\R)
    \end{align*}
    with $C(W_\eta,v,p^{max},\gep)=W_\eta(0) (\lVert {v_\gep} \rVert +\lVert {v_\gep}' \rVert )  \lVert\rho\rVert$.
\end{lemma}
\begin{proof}
For the usage of approach \cite[Thm. 3.4]{Friedrich2018} the crucial prerequisites, any stochastic velocity need to fulfill, is the existence of an at least almost everywhere existing, non increasing derivative. 
Further, in order to derive mean-value-theorem-based equalities, the Lipschitz continuity must hold and the zero-sets on which a derivative might not exist needs to be excluded.
For our specific velocity function the relevant prerequisites are defined and discussed in Section \ref{subsec:ana_v_e}.
The full proof is provided in Appendix as Section \ref{lem:BV_app}.
\end{proof}

\begin{remark} \label{rem:BV_space_indep_e} 
    Independently to the above, we can derive an $\gep$-independent bound as it holds
    \begin{align*}
        \text{TV}(\rho^{\Delta x}(T,\cdot);R) 
        &\leq \exp \Bigl(T \underbrace{2 W_\eta(0) \lVert v' \rVert  \lVert\rho\rVert}_{
            =:\, \tilde{C}(W_\eta,v,\rho^{max})}\Bigr) \text{TV}(\rho_0;\R).
    \end{align*}
    For details on the derivation of this bound, see the extended Remark \ref{rem:BV_space_indep_e_app} of Section \ref{lem:BV_app}.
\end{remark}
Given Lemma \ref{lem:BV_space} or optionally Remark \ref{rem:BV_space_indep_e}, we now present the
TV bound in space and time.

\pagebreak 

\begin{lemma}[BV estimate in space and time]\label{lem:BV_space_time}\text{ }\\
    Let $\rho_0(x)$ as in  (\ref{eq:CP_BV}) and $\gep(t,\omega)$ as in (\ref{eq:construction_eps}). 
    Further, let the usual assumptions on $W_\eta$ and $v$ hold.
    If $\rho^{\Delta x}$, resp. $\rho_j^n$, is constructed by (\ref{eq:SCHEM_sNV})
    with the CFL condition (\ref{eq:sCFL}), then for all $T>0$ we have the $\gep$-dependent 
    discrete space and time BV estimate:
    \begin{align*}
        \text{TV}(\rho^{\Delta x};\R \times [0,T]) \leq T \exp \bigl(T C \bigr) 
                \Bigl( 1+W_\eta(0) \lVert v' \rVert \lVert \rho \rVert + \lVert v_\gep \rVert\Bigr)
                \text{TV}(\rho_0;\R),
    \end{align*}
    with $C(W_\eta,v,p^{max},\gep)$ as in Lemma \ref{lem:BV_space} or $\tilde{C}(W_\eta,v,\rho^{max})$ as in Remark \ref{rem:BV_space_indep_e}.
\end{lemma}
\begin{proof}
    We once again rely on the redefined derivative for the possibly only pointwise existing derivative of $v_\gep$. 
    For the estimate it is crucial that for any stochastic velocity function, neither the derivative nor its numerical discretization surpasses the absolute value of $v'$. 
    The same must hold true for the convolution itself, i.e. $|V_{\gep,j}^n|\leq \lVert v_\gep^n \rVert$. 
    We present the full proof in Section \ref{lem:BV_app}.
\end{proof}

With the above, we have everything at hand to show the existence of some convergent subsequence of $\rho^{\Delta x}$
by Helly's Theorem (\cite[5.6]{Eymard2000},\cite[2.19]{Friedrich2021_diss}). 
Yet, it is left to show that this limit is the weak entropy solution. 
We introduce the notations $a\wedge b := \max\{a,b\}$, $a \vee b := \min\{a,b\}$ and show that the stochastic numerical densities $\rho_j^n$ satisfy a discrete entropy inequality.

\begin{lemma}[Discrete entropy inequality]\label{lem:discrete_enropy_ineq}\text{ }\\
    Let $\gep(t,\omega)$ as in (\ref{eq:construction_eps}).
Further, let the usual assumptions on $W_\eta$ and $v$ hold.
If $\rho^{\Delta x}$, resp. $\rho_j^n$, is constructed by (\ref{eq:SCHEM_sNV})
with the CFL condition (\ref{eq:sCFL}), then,
for any $c \in \R$ and for all $n \in \{0,\dots,N_T\},\, j\in \Z$, the following discrete entropy inequality holds true
    \begin{align}\label{eq:discrete_enrtropy_ineq}
        |\rho_j^{n+1}-c| \leq |\rho_j^n-c| &- \lambda \left( H_{j+1/2}^n(\rho_j^n) - H_{j-1/2}^n(\rho_{j-1}^n)  \right) \nonumber\\
                                        &-\lambda \text{\emph{ sign}}(\rho_j^{n+1}-c) \left( F_{j+1/2}^n(c)-F_{j-1/2}^n(c) \right), 
    \end{align}
where 
    \begin{align*}
        H_{j+1/2}^n(u):= F_{j+1/2}^n(u \wedge c)-F_{j+1/2}^n(u \vee c).
    \end{align*}
\end{lemma}

\begin{proof}
    Adapting from  \cite[2.25]{Friedrich2021_diss}, we initially require the non-negativity of $v_\gep$ for all possible $\gep$ values, as ensured by the maximum.
    Further, this proof relies on the spatial differentiability of the numerical flux $F_{j+1/2}^n(\rho)$  at each time step $n$.
    We have ensured this, by constructing our error term $\gep(t^n,\omega)$ constant in the spatial dimension.   
    The detailed  proof can be found in Section \ref{lem:discrete_enropy_ineq_app}.
\end{proof}

It rests to examine the limiting behavior of the derived discrete entropy inequality (\ref{eq:discrete_enrtropy_ineq}).
Given the non-local nature of our scheme, the classical argumentation regarding the numerical limit of $(\Delta t,\Delta x) \rightarrow 0$, via the Lax Wendroff Theorem becomes inadequate. Consequently, employing nonlocal theory, the time integrability of the stochastic flux function becomes crucial in the presented approach.

\begin{lemma}[Convergence to entropy solution]\label{lem:convergence}\text{ }\\
    Fix any $\omega \in \Omega$ and
    assume that $\rho=\rho(t,x)(\omega)$ with $\rho(t,\cdot)(\omega) \in$ BV$(\R)$ 
    for some $t \in [0,T]$ is the $L^1_{loc}$-limit (as in Helly's Theorem e.g. \cite[5.6]{Eymard2000}) of $\rho^{\Delta x}$, constructed by
    (\ref{eq:SCHEM_sNV}), and exists. Let the assumptions on $W_\eta$ (\ref{rem:ass_W}) and
     $v$ (\ref{rem:ass_v_ND}) hold. Then for any $c \in \R$, $\phi \in C_0^1([0,T)\times \R;\R^+)$ 
     $\rho$ is a weak entropy solution in the sense of Def. \ref{def:nonlocal_weak_entropy_sol}.
\end{lemma}

\begin{proof}
    As we adapt the proof of \cite[2.2.6]{Friedrich2021_diss}, we ensure that the occurring integrals, especially
        \begin{align*}
            \int_{0}^{T} \int_{\R} \text{sign}(\rho-c) 
            \Bigl(f_{{\gep}}(t,x,\rho)-f_{{\gep}}(t,x,c)\Bigr) \,dx  \,dt,
        \end{align*}
    are well posed. As discussed, we achieve this by using the piecewise constant error term $\gep(t,\omega)$, with parameter $\delta_R$.
    Furthermore, we must address the time dependence of $v_\gep^n$, which we solve by additionally bounding the occurring terms in time, leveraging the Lipschitz-continuity and the fact that $\gep$ is of bounded total variation.
   The complete proof is given in Section \ref{lem:convergence_app}.
\end{proof}

\subsection{Existence and Uniqueness of the sNV model} \label{subsec:ex_uniq} 
We now have established all prerequisities for the main theorem on the existence of solutions to the sNV model. 
All our findings, including the yet-to-be-shown uniqueness, are  encapsulated in the central theorem of this contribution. 

\begin{satz}[Existence, uniqueness and properties of (sNV)]\label{thm:ex_uniq_sNV_extended}\text{ }\\
   Fix any $\omega \in \Omega$ and let $\rho_0(x)$, $\gep(t,\omega)$ as in (\ref{eq:CP_BV}), (\ref{eq:construction_eps}).
     Further, let the usual assumptions on $W_\eta$ and $v$  hold. Then, for 
     any $T>0$ a weak entropy solution $\rho(t,x)(\omega)$, in the sense of Def. \ref{def:nonlocal_weak_entropy_sol}, to 
     the Cauchy Problem of (\ref{eq:sNV}), i.e.
     \begin{align*}
        \begin{cases}
			\partial_t \rho(t,x) + \partial_x f_{\gep}(t,x,\rho(t,x))=0, & (t,x) \in (0,T] \times \mathbb{R},  \\ 
			\rho(0,x) =\rho_{0}, & x \in \mathbb{R},
		\end{cases} 	
     \end{align*}
     with $f_{\gep}(t,x,\rho)=\rho \bigl(W_\eta * v_\gep(\rho,\cdot)\bigr)(t,x)$, exists and is unique.\\
    Further, it holds for all $\omega \in \Omega$:
    \begin{enumerate}
        \item Maximum principle: 
        \begin{align*} 0 \leq \inf_{\R} \rho_0(x) \leq \rho(t,x)(\omega) \leq \sup_{\R} \rho_0(x) \leq \rho^{max} \quad \forall\, t \in [0,T]. \end{align*}
        \item $L^1$-conservation: 
        \begin{align*} {\lVert (\rho(t,\cdot)) (\omega)\rVert}_{L^1(\R)}={\lVert \rho_0 \rVert}_{L^1(\R)}\quad \forall\, t \in [0,T]. \end{align*}
        \item TV bounds: 
        \begin{align*}
            \text{TV}(\rho(T,\cdot);\R) &\leq \exp \Bigl(T C(W_\eta,v,\rho^{max},\gep) \Bigr) \text{TV}(\rho_0;\R), \quad T>0\\
            \text{TV}(\rho(T,\cdot);\R) &\leq \exp \Bigl(T \tilde{C}(W_\eta,v,\rho^{max}) \Bigr) \text{TV}(\rho_0;\R), \quad T>0\\
            \text{TV}(\rho;\R \times [0,T]) &\leq T \exp \bigl(T C \bigr) 
                \Bigl( 1+W_\eta(0) \lVert v' \rVert \lVert \rho \rVert + \lVert v_\gep \rVert\Bigr) 
                \text{TV}(\rho_0;\R) , \quad T>0.
        \end{align*}
    \end{enumerate}
\end{satz}
\begin{proof}
    We commence by showing that the prerequisites for Helly's Theorem are satisfied. 
    Due to the discrete maximum principle (Lemma \ref{lem:maxP}) we obtain a $L^{\infty}(\R)$ bound, given by $\rho_0$, as
    well as an TV bound in space and time by Lemma \ref{lem:BV_space_time}. 
    As we fix $\rho(t,x):=\rho_0 \; \forall t<0$ numerically, it is sufficient to show a TV estimate over $[0,T)\times\R$ instead of $(-T,T)\times\R$, since
    \begin{align*}
        \text{TV}(\rho;(-T,T)\times\R)\leq K_T + T\cdot \text{TV}(\rho_0;\R)=:\tilde{K}_T,
    \end{align*}
    assuming $K_T$ is the TV estimate over $[0,T)\times\R$ from Lemma \ref{lem:BV_space_time}.
    Hence, we can restrict to $[0,T)\times\R$ to apply Helly's Theorem, whilst keeping the possibly higher TV estimate in mind. 
    Therefore the assumptions on Helly's Theorem are valid, such that we obtain for any fixed $\omega \in \Omega$ the existence of a sub-sequence of $\rho^{\Delta x}$ converging to some limiting density $\rho(t,x)(\omega)$.
    Due to the assumptions on $\rho_0$ and the discrete $L^1$-conservation (Corollary \ref{cor:L1_conservation}), it additionally holds that
    \begin{align*}
        \rho(\cdot,\cdot)(\omega) \in L^1([0,T)\times \R ; \R), \quad \forall \omega \in \Omega. 
    \end{align*}
    Next, by our assessment regarding the convergence of our numerical scheme (Lemma \ref{lem:convergence}),
    $\rho$ is not only some $L^1_{loc}$-limit of $\rho^{\Delta x}$ but 
    satisfies the weak entropy condition in the sense of Definition \ref{def:nonlocal_weak_entropy_sol}. 
    Thereby a weak entropy solution to the sNV model exists for every realization of the random velocities 
    and is unique as we show in the following Theorem \ref{thm:uniqe_sNV}. \\

    The additional claims (1.-3.) then follow from their discrete counterparts, i.e. Lemmas
    \ref{lem:maxP}, \ref{cor:L1_conservation}, \ref{lem:BV_space}, \ref{rem:BV_space_indep_e}, \ref{lem:BV_space_time}, 
    and the proven convergence of the numerical scheme itself.
\end{proof}
\begin{remark} \label{rem:time_cts_rho}
Additionally, due to equation (\ref{eq:Delta-t}) it holds that
    \begin{align*} 
        \rho(\cdot,x)(\omega) \in C((0,T]) \quad \forall x \in \R, \, \forall \omega \in \Omega , 
    \end{align*} 
    despite the fact that the velocities $v_\gep(\rho,t)$ are not only time-dependent,
    but also lack continuity over time, as $\gep(t,\omega)$ does not possess it.   
\end{remark}

Note, that every realization of $\gep(t,\omega)$ leads to a different solution $\rho(t,x)(\omega)$.
Hence, the uniqueness contained in our central result, has to be understood per given state  $\omega $ of the world.  
The missing piece for Theorem \ref{thm:ex_uniq_sNV_extended} is then as follows.

\begin{satz}[Uniqueness of entropy solutions to (\ref{eq:sNV})]\label{thm:uniqe_sNV} \text{ }\\
    Let the assumptions on $v$ and $W_\eta$ hold. 
    Let
    $\rho,\sigma$ be two weak entropy solutions to (\ref{eq:sNV}) with initial data $\rho_0,\sigma_0$ respectively. 
    Then, for any $T>0$, it holds
    \begin{align*}
        {\left\lVert \bigl(\rho(t,\cdot)-\sigma(t,\cdot)\bigr)(\omega) \right\rVert}_{L^1(\R)} \leq \exp(K_\gep(\omega) T) 
        {\left\lVert \rho_0-\sigma_0 \right\rVert}_{L^1(\R)} \quad \forall t \in [0,T], \; \forall \omega \in \Omega,
    \end{align*}
    with $0\leq K_\gep(\omega) \leq K \; \forall \omega \in \Omega$ given in Section \ref{thm:uniqe_sNV_app}.
\end{satz}
\begin{proof}
    In order to translate the ideas from the respective concepts of the NV model (see e.g. \cite[2.4]{Friedrich2018}), our proof relies on the bounds of $v_\gep$ in the norm derived in Section \ref{subsec:stoch_v_e} as well as on the properties of its derivative discussed in Section \ref{subsec:ana_v_e}.
    The complete proof is provided in Section \ref{thm:uniqe_sNV_app}.
\end{proof}

The above relies on the classical entropy condition to filter the unique solution to (\ref{eq:sNV}). 
However, as we have seen, numerical evidence indicates that the characteristics do not cross, which gives hope that a stronger uniqueness results might hold true as well. 

 \begin{remark} \label{rem:measurability_of_rho}
 Since our  proof is done for fixed but arbitrary $\omega \in \Omega$, only the measurability and boundedness of the noise process are relevant. 
 However, by using Helly's Theorem for a fixed $\omega \in \Omega$,  the measurability of the map
\begin{align*}
	\rho: \Omega \rightarrow C([0,T]; L^1(\R)), 
\end{align*}
remains open, and we will address this problem in our future research. Thus, quantities as $\mathbb{E} \left[ \rho(t,x) \right]$ are a-priori not well-defined. Yet, this does not affect the Monte-Carlo simulations in our work, since  on the one  hand they  rely on a discrete approximation scheme, whose quantities are well-defined random variables by construction, see Remark \eqref{num-mb} in Section \ref{sec:num schemes nonloc_stoch}, and on the other hand  they depend on the used pseudo random number generator, which corresponds (at best) to a discrete approximation of the uniform distribution.  See also the following Remark \ref{rem:stochastic_everywhere}.
\end{remark}

\begin{remark}\label{rem:stochastic_everywhere}
    If we replace the continuous uniform distribution on $[-\tau,\tau]$ by a discrete uniform distribution, e.g. with the uniform distribution on $$ \mathcal{M}= \left \{   \tau \frac{i}{M}: \, i=-M, \ldots, M \right \},$$ then
	we can work on a finite-dimensional probability space and the map
	\begin{align*}
		\rho: (\Omega,\mathcal{A}) \rightarrow C([0,T]; L^1(\R)), 
	\end{align*}
	would be trivially measurable for $\mathcal{A}=\mathfrak{P}(\Omega)$.
	\end{remark}
 
	\begin{remark} In \cite{Mishra2012} a top-down approach is used by considering
	\begin{align*}
 		\partial_t u(\omega;x,t) = \partial_x f(\omega; u(\omega;x,t)), \quad u(\omega;x,0) =u_0(\omega;x),     \qquad t >0, \; x \in \mathbb{R}^d.
	\end{align*} Here the random flux $f: \Omega \rightarrow  C^{1}(\mathbb{R}^d; \mathbb{R})$ is assumed to be bounded and measurable, and the randomness is incorporated via a Karhunen-Lo\`eve expansion. Hence, similar to our approach, the noise is introduced via a countable set of random variables. 
 Yet, our bottom-up approach does not fit into the given framework, since our model incorporates a time-dependent random flux function $f_\gep(t,x,\rho)$.
\end{remark}

\section{Numerical results} 
\label{sec:num_results}
Having established the existence and uniqueness of solutions, we proceed to present numerical results for the sNV model.
In doing so, we emphasize the key properties of the proposed model, compare them to its deterministic base  
(\ref{eq:NV}) and comment on the exhibited behavior from a modeling perspective.
To conduct our analysis, we employ our numerical scheme (\ref{eq:SCHEM_sNV}), given the velocity function $v(\rho)=1-\rho^2$ and the kernel $W_\eta^{conc.}$. 
Further, we set $\Delta x=0.3\cdot 10^{-2}$ and calculate $\Delta t$ according to the CFL condition (\ref{eq:sCFL}), depending on $\eta$ and $\tau$.
While presenting some actual realizations in grey, our focus lies on the mean behavior of the sNV model. 
In the following, we determine this key quantity and its empirical quantiles in accordance with Remark \ref{rem:measurability_of_rho}, i.e. by Monte-Carlo simulations.

\subsection{Probabilistic densities and their mean behavior  }\label{subsec:num_comp_to_NV}
    \begin{figure}[htb]
        \centering
        \begin{subfigure}{.5\textwidth}
            \centering
            \includegraphics[width=1.004\linewidth]{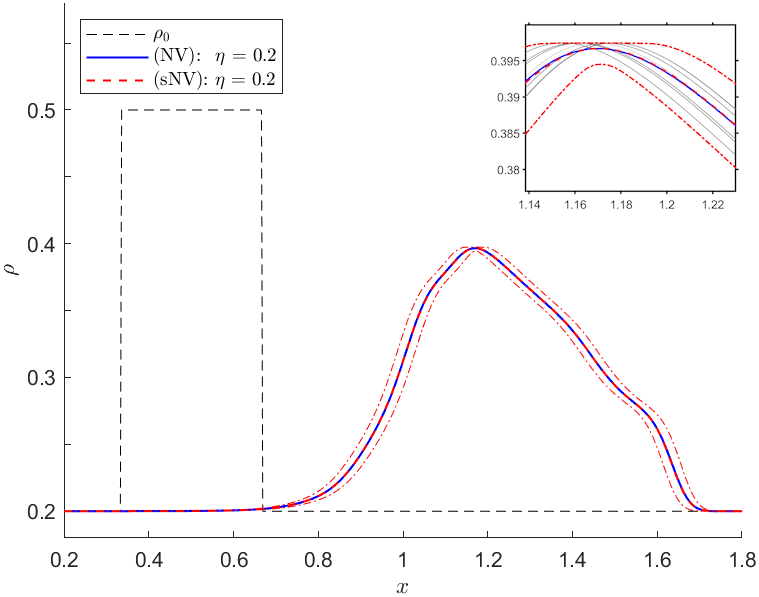}
            \caption{Low initial density, time horizon $t=1$}
            \label{fig:NV_weak_comp}
          \end{subfigure}%
          \begin{subfigure}{.5\textwidth}
            \centering
            \includegraphics[width=1.0\linewidth]{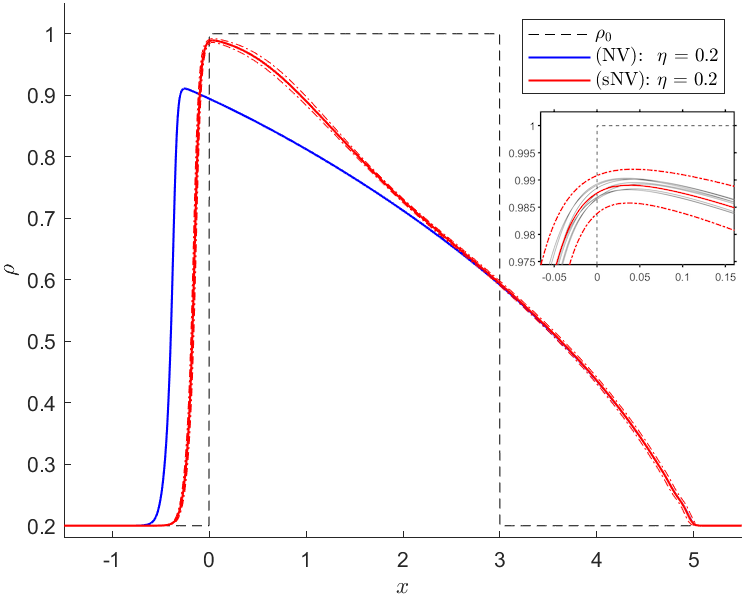}
            \caption{High initial density, time horizon $t=2$}
            \label{fig:NV_strong_comp}
          \end{subfigure}
        \caption{Mean of $N=10^4$ (\ref{eq:sNV}) realizations compared to (\ref{eq:NV}) for $\tau=0.5$. The pointwise $5\%$ and $95\%$ quantiles are given by the dotted lines. In the magnification some realizations of (\ref{eq:sNV}) are given in grey.}
        \label{fig:NV_comp}
    \end{figure}
We start our numerical assessment with an example where a low initial congestion with respect to $\tau$ and $\eta$ is chosen, such that
    \begin{align*}
    \rho_0(x)  <  v^{-1}(\tau)-\eta = \sqrt{1-0.5}-0.2  \quad \forall x \in \R
\end{align*} 
holds.
Due to $\gep(0)=0$ we have
\begin{align*}
    \E[V_\gep(0,x)](\rho_0(x))=V(0,x)(\rho_0(x)) \quad \forall x \in \R
\end{align*}
and by the maximum principle (Lem. \ref{lem:maxP}) and Lemma \ref{lem:exp_var_v} it follows
\begin{align*}
    \E[V_\gep(t,x)](\rho(t,x))=V(t,x)(\rho(t,x)) \quad \forall (t,x) \in [0,T]\times \R.
\end{align*}
Given these initial conditions, our numerical results, as depicted in Figure \ref{fig:NV_weak_comp} and Table \ref{tab:err_against_NV}, indicate that the missing bias of the velocity function does translate to the mean behavior of (\ref{eq:sNV}).
Hence, given this sufficiently small congestion, the sNV model empirically achieves mean consistency with its deterministic basis (\ref{eq:NV}) at least up to an error term.

\begin{table}[htb]
    \centering
\begin{tabular}{ll|ccc}\hline
                                    &                 & $N=10^3$ & $N=10^4$  &$N= 5\cdot 10^4$\\ \hline 
        $\Delta x=0.3 \cdot 10^{-2}$&$L^1$-dist.       & $0.046717$ & $0.040547$ &$0.040294$        \\
                                    &$L^2$-dist.       & $0.000011$ & $0.000008$ &$0.000008$        \\ 
                                    &$L^\infty$-dist.  & $0.000553$ & $0.000496$ &$0.000489$        \\ \hline
        $\Delta x= 10^{-3}$         &$L^1$-dist.       & $0.075139$ & $0.075888$ &$0.063205$        \\
                                    &$L^2$-dist.       & $0.000018$ & $0.000018$ &$0.000011$        \\ 
                                    &$L^\infty$-dist   & $0.000711$ & $0.000711$ &$0.000524$        \\ \hline
        
\end{tabular}
    \caption{Distances between the mean of (\ref{eq:sNV}) and (\ref{eq:NV}) of Figure \ref{fig:NV_weak_comp}.}
    \label{tab:err_against_NV}
\end{table}

This behavior is further emphasized by the fact that every probabilistic realization is not only well posed (Thm. \ref{thm:ex_uniq_sNV_extended}) but also admits a similar behavior as (\ref{eq:NV}), i.e. they are smooth to the same degree, they have a comparable height, and they have moved similarly over the spatial domain. 
Therefore our proposed model allows us to fit the deterministic basis to the mean of given empirical data and tweak the stochastic parameters to explain observations further offside the observed mean. 
A rigorous proof for the illustrated behaviour alongside the development of an understanding how the expectation of such SCL is to be understood analytically, is planned for future research. 

To emphasize the influence of the initial data to the behavior of the model, we increase the height and domain of $\rho_0$'s maxima and plot the results in Figure \ref{fig:NV_strong_comp}. 
Given such data, the sNV model exhibits a notably distinct mean behavior than the NV model.
By our assessment of the stochastic velocity function in Section \ref{subsec:stoch_v_e}, we attribute  this phenomenon to the altered velocity of the interacting vehicles.
Compared to (\ref{eq:NV}), vehicles detecting a congestion ahead, employ an on average higher velocity, resulting in a slower dissipation of the congestion. 
Once the downstream densities are low enough or have dissolved appropriately, (\ref{eq:sNV}) smoothly transitions back to becoming mean consistent as before.
This mean deviation, aligns well with empirical observations presented in \cite[Fig. 1]{Wang2011}, where the authors state, that a deterministic density-velocity relation seems to break down especially for high densities. 
Thus, our proposed model successfully replicates such deviations, including increased uncertainty at high densities, while maintaining consistency in regions of low density.

\begin{remark}
During our assessment of the characteristics, it was observable that for any given initial density the mean appears to be well captured by the NV model when utilizing the expected velocity $\bar{v}$.
As initial numerical experiments have shown, this does translate to the densities as well. 
While such result would allow us to capture the expected behavior for any given initial density, further research is necessary to validate this claim.
\end{remark}

\subsection{Influence of parameters}\label{subsec:num_paramvar}

Having seen the significance of the initial densities alongside the overall model behavior, we now proceed to demonstrate the modeling capabilities by varying the two central parameters $\tau$ and $\eta$.\\

\begin{figure}[htb]
    \centering
    \begin{subfigure}{.5\textwidth}
        \centering
        \includegraphics[width=1.0\linewidth]{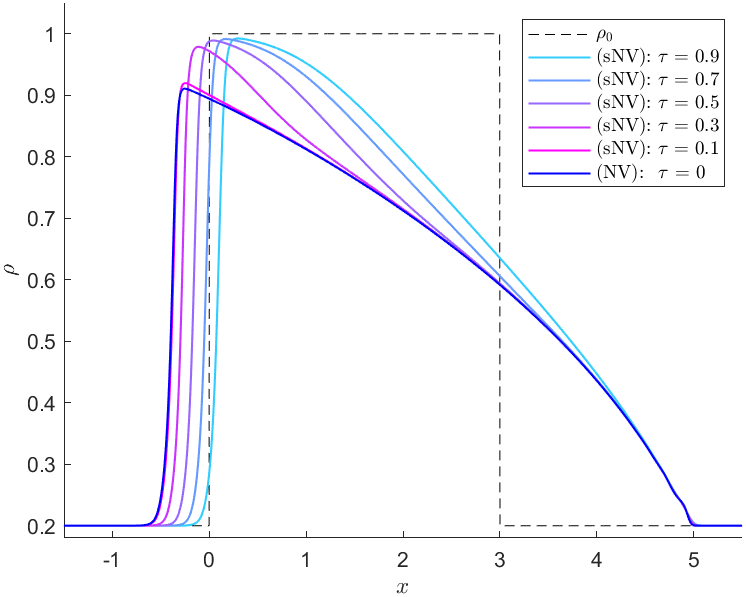}
      \end{subfigure}%
      \begin{subfigure}{.5\textwidth}
        \centering
        \includegraphics[width=1.0\linewidth]{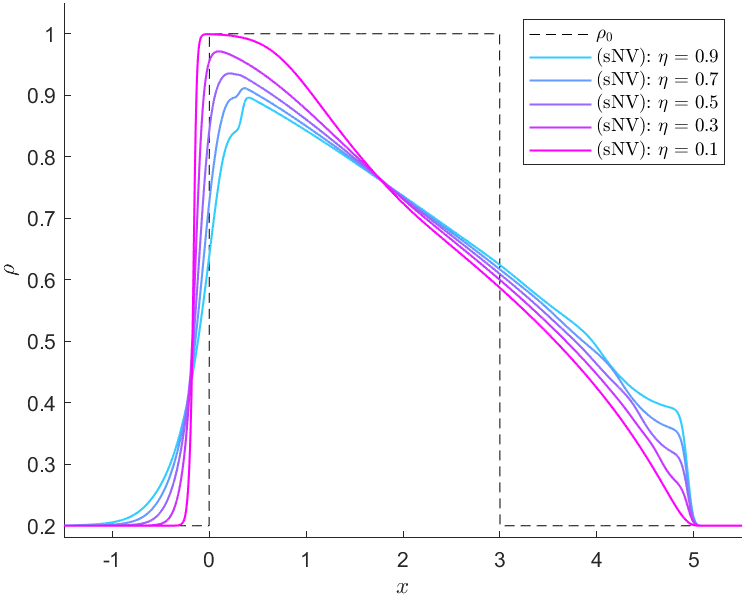}
      \end{subfigure}
    \caption
    {Mean of $N=10^4$ (\ref{eq:sNV}) realizations for each $\tau$, ($\eta=0.2$) and each $\eta$, ($\tau=0.5$), at $t=2$, given high initial density.}
    \label{fig:big_bump_multi_tau_eta}
\end{figure}

Alternating $\tau$ in Figure \ref{fig:big_bump_multi_tau_eta} (left), we note a nonlinear coupling of (\ref{eq:sNV})'s mean with respect to the strength of the error term.
Furthermore, the behavior is monotone with the increase or decrease of $\tau$, and as $\tau \rightarrow 0$ we empirically observe convergence of (\ref{eq:sNV}) to (\ref{eq:NV}). \\

Considering a fixed error strength $\tau$ and varying the look ahead distance $\eta$ in Figure \ref{fig:big_bump_multi_tau_eta} (right), the sNV model exhibits a comparable behavior as the NV model (see e.g. \cite[2.4]{Friedrich2021_diss}). 
Especially for the extreme cases of $\eta \rightarrow 0$ and $\eta \rightarrow \infty$, Figure \ref{fig:big_bump_multi_tau_eta} alongside additional experiments suggest the convergence to the solution of a LWR model, using $\bar{v}$, as well as to a Transport Problem, with a different characteristic speed as for (\ref{eq:NV}). 
However, in comparison to the results of \cite[2.4]{Friedrich2021_diss} our preliminary results need additional theory on the convergence obtained on non-compact supports and once again rely on the role of the expected velocity $\bar{v}$ with respect to the overall behavior.

\section{Conclusion and outlook} \label{sec:conclusion}

In this work, we proposed a novel traffic flow model incorporating stochastic nonlocal velocities and analyzed it from both theoretical and numerical perspectives. 
We rigorously proved the existence and uniqueness of weak entropy solutions and developed an effective numerical scheme as well as an implementation of the model.
Our analysis revealed that the introduced noise does not significantly destabilize the system, making the proposed model a promising candidate for real-world applications.
Furthermore, we investigated the conditions under which the mean dynamics of the perturbed system align with the deterministic base model, emphasizing the influence of both nonlocal and stochastic parameters in relation to the initial density.
We specifically demonstrated that the mean behavior of the stochastic model deviates from the deterministic model, particularly during periods of high traffic densities. 
This observation is consistent with empirical data, suggesting that a deterministic relationship between density and velocity may not adequately describe traffic under congested conditions.
Additionally, we introduced and provided initial results supporting a novel theory on how the mean behavior could be captured for any set of initial conditions.

For future research, we suggest further developing and integrating the proposed model into a comprehensive theoretical framework,
especially defining how expected solutions and the variance of scalar conservation laws with stochastic fluxes are to be understood.

\section*{Acknowledgments} 
S.G. acknowledges financial support of the German Research Foundation (DFG) within the projects GO1920/11-1 and GO1920/12-1.

\bibliographystyle{siam}
\bibliography{quellen,literature_ext}

\appendix 
\section{Appendix} \label{sec:App}

\subsection{Detailed proofs}\label{App_proofs} 

To emphasise the stochastic influence alongside the time dependence of $v_\gep$, necessary changes and adjustments to the \emph{proofs} of the NV model are highlighted in \textcolor{blue}{blue color}.

\subsubsection{Discrete maximum principle}\label{lem:maxP_app} 
The detailed proof regarding Lemma \ref{lem:maxP} is as follows.
\begin{proof}
    As in \cite[3.3]{Friedrich2018}, we prove the claim by induction in time, whilst applying our findings on 
    $v_\gep$ and $v_\gep'$ from Section \ref{subsec:ana_v_e}.
    \begin{itemize}[noitemsep]
        \item For $n=0$ the claim is trivial.
        \item Suppose (\ref{eq:max_p_*}) holds for a fixed but arbitrary $n \in \N$.
        \item Now, if we apply the definition of the discrete convolution, it holds
        \begin{align} 
            V_{\gep,j-1}^n-V_{\gep,j}^n 
            &= \sum_{k=1}^{N_\eta-1}
            \underbrace{(\gamma_{k}-\gamma_{k-1}) }_{\leq 0}
                 v_\gep^{n}(\rho_{j+k}^{n}) 
            \underbrace{-\gamma_{N_\eta-1}}_{<0}
                v_\gep^{n}(\rho_{j+N_\eta}^{n})
            +\gamma_{0}
                v_\gep^{n}(\rho_{j}^{n}) \label{eq:max_p_V_eq}\\ 
            &\leq \sum_{k=1}^{N_\eta-1}
            (\gamma_{k}-\gamma_{k-1}) 
                 v_\gep^{n}\Bigl(\sup_{j \in \Z} \rho_j^0\Bigr) 
            -\gamma_{N_\eta-1}
                v_\gep^{n}\Bigl(\sup_{j \in \Z} \rho_j^0\Bigr) 
            +\gamma_{0}
                v_\gep^{n}(\rho_{j}^{n}) \nonumber \\
            &=\gamma_{0} \left(v_\gep^{n}(\rho_j^n) -v_\gep^{n}\Bigl(\sup_{j \in \Z} \rho_j^0\Bigr)    \right) \nonumber\\
            &\leq
            \gamma_0 {\lVert  {v_\gep^n}' \rVert}_{L^\infty} 
            \Bigl(\underbrace{\sup_{j \in \Z} \rho_j^0 -\rho_j^n}_{\geq 0}\Bigr), \label{eq:V-V_bound}
        \end{align}
        using the \textcolor{blue}{monotonicity on $v_\gep^n$, i.e. $({v_\gep^n})' \leq 0$ 
        }, along with the induction hypothesis.
        The second equality then follows from a telescoping sum argument.
        \textcolor{blue}{Lastly we used the 
        Lipschitz-continuity of $v_\gep^n$.}
        If we multiply this by $\sup_{j \in \Z} \rho_j^0$, subtract $V_{\gep,j}^n\rho_j^n$ on
        both sides and rearrange the equations, we derive
        \begin{align}
            V_{\gep,j-1}^n \sup_{j \in \Z} \rho_j^0 - V_{\gep,j}^n\rho_j^n 
            & \leq \Bigl(\gamma_0  {\lVert {v_\gep^n}' \rVert}  \sup_{j \in \Z} \rho_j^0  + V_{\gep,j}^n \Bigr) 
            \Bigl(\sup_{j \in \Z} \rho_j^0 -\rho_j^n\Bigr)\nonumber\\
            &\leq \Bigl(\gamma_0 {\lVert {v_\gep^n}' \rVert} \rho^{max} + \lVert v_\gep^n \rVert \Bigr) 
            \Bigl(\sup_{j \in \Z} \rho_j^0 -\rho_j^n\Bigr) \label{eq:Vsup-V}
        \end{align}
        and thus, by the definition of the scheme and again the induction hypothesis (IH)
        \begin{align*}
            \rho_j^{n+1} 
            &= \rho_{j}^n+\lambda \Bigl(F_{j-1/2}^n(\rho_{j-1}^n)-F_{j+1/2}^n(\rho_{j}^n)  \Bigr)\\
            \overset{\text{(IH)}}&{\leq} \rho_{j}^n+\lambda \Bigl( V_{\gep,j-1}^n \sup_{j \in \Z} \rho_j^0 -  V_{\gep,j}^n\rho_j^n \Bigr)\\
            \overset{(\ref{eq:Vsup-V})}&{\leq} 
            \rho_{j}^n+
            \underbrace{\textcolor{blue}{\lambda  \Bigl(\gamma_0{\lVert {v_\gep^n}' \rVert} \rho^{max}  + \lVert v_\gep^n \rVert \Bigr)}}_{=:\Lambda} 
            \Bigl(\sup_{j \in \Z} \rho_j^0 -\rho_j^n\Bigr)\\
            &= (1-\Lambda) \rho_{j}^n + \Lambda \sup_{j \in \Z} \rho_j^0 \;
            \leq \sup_{j \in \Z} \rho_j^0,
        \end{align*}
        since $1-\Lambda > 0$ due to the \textcolor{blue}{adapted CFL condition}.\\
        For the left inequality in (\ref{eq:max_p_*}) we can proceed as above, changing multiple signs
        to get to
        \begin{align*}
            \rho_j^{n+1} 
            &\geq \rho_{j}^n+ \Lambda \Bigl(\inf_{j \in \Z} \rho_j^0 - \rho_{j}^n \Bigr) \geq \inf_{j \in \Z} \rho_j^0. \qedhere
        \end{align*} 
    \end{itemize}
\end{proof}
\pagebreak 
\begin{remark}
    We can show a different, on $\Delta x$ dependent, upper bound of the differences (\ref{eq:max_p_V_eq}) by
    using $\gamma_k \leq W_\eta(0) \Delta x, \; \forall k=0,\dots,N_\eta-1$:
    \begin{align}\label{eq:bound_on_diff_deltax}
        |V_{\gep,j}^n-V_{\gep,j-1}^n| 
        &= \left\lvert  \sum_{k=0}^{N_\eta-1} \gamma_k v_\gep^n(\rho_{j+k+1}^n) -
                        \sum_{k=0}^{N_\eta-1} \gamma_k v_\gep^n(\rho_{j+k}^n)  \right\rvert \nonumber \\ 
        &\leq W_\eta(0) \Delta x \left( 
             \sum_{k=0}^{N_\eta-1} \left\lvert v_\gep^n(\rho_{j+k+1}^n)-v_\gep^n(\rho_{j+k}^n) \right\rvert \right) \nonumber \\
        &\leq W_\eta(0) \Delta x \left( 
            \sum_{k=0}^{N_\eta-1}  \lVert v_\gep \rVert \lVert \rho \rVert \Delta x  \right) \nonumber \\
        &= W_\eta(0) \Delta x \Bigl( \left\lfloor \eta/\Delta x \right\rfloor \lVert v_\gep \rVert \lVert \rho \rVert \Delta x  \Bigr) \nonumber \\
        &\leq \Delta x  W_\eta(0)  \eta \lVert v_\gep \rVert \lVert \rho \rVert .
    \end{align}
\end{remark}

\subsubsection{BV estimates} \label{lem:BV_app} 
To prove the spatial bound of the total variation, i.e. Lemma \ref{lem:BV_space}, we proceed as follows.
\begin{proof}
    We use the method described in \cite[3.4]{Friedrich2018} and define $\Delta_{j+k-\frac{1}{2}}^n:=\rho_{j+k}^n-\rho_{j+k-1}^n$.\\
    Then, if we apply the scheme twice
    \begin{align*}
        \Delta_{j+\frac{1}{2}}^{n+1}
        &=\rho_{j+1}^{n+1}-\rho_{j}^{n+1}\\
        &=\rho_{j+1}^{n}-\rho_{j}^{n}
            -\lambda \Bigl( F_{j+\frac{3}{2}}^{n}(\rho_{j+1}^{n})-F_{j+\frac{1}{2}}^{n}(\rho_{j}^{n}) \Bigr)
            +\lambda \Bigl( F_{j+\frac{1}{2}}^{n}(\rho_{j}^{n})-F_{j-\frac{1}{2}}^{n}(\rho_{j-1}^{n}) \Bigr)\\
        &=\Delta_{j+\frac{1}{2}}^{n}
            -\lambda \Bigl( V_{\gep,j+1}^{n} \rho_{j+1}^{n} -2 V_{\gep,j}^{n} \rho_{j}^{n} + V_{\gep,j-1}^{n} \rho_{j-1}^{n}     
                     \Bigr)\\
        \overset{\pm0}&{=} \Delta_{j+\frac{1}{2}}^{n}
            -\lambda \Bigl[ V_{\gep,j+1}^{n} (\rho_{j+1}^{n}-\rho_{j}^{n})
                            -V_{\gep,j-1}^{n} (\rho_{j}^{n}-\rho_{j-1}^{n})
                            +\rho_j^n \Bigl( \underbrace{V_{\gep,j+1}^{n}-2V_{\gep,j}^{n}+V_{\gep,j-1}^{n}}_{=:(\text{\textasteriskcentered{}})} \Bigr)\Bigr].
    \end{align*}
    Now we can use (\ref{eq:max_p_V_eq}) twice to obtain
    \begin{align*}
        (\text{\textasteriskcentered{}}) &=
         -\Bigl(V_{\gep,j}^{n}-V_{\gep,j+1}^{n}\Bigr)+\Bigl(V_{\gep,j-1}^{n}-V_{\gep,j}^{n}\Bigr)\\
            &= -\gamma_0v_\gep^{n}(\rho_{j+1}^{n}) 
                                - \sum_{k=1}^{N_\eta-1}(\gamma_{k}-\gamma_{k-1})  v_\gep^{n}(\rho_{j+1+k}^{n})
                                +\gamma_{N_\eta-1} v_\gep^{n}(\rho_{j+1+N_\eta}^{n})\\
            &\quad\;+\gamma_0v_\gep^{n}(\rho_{j}^{n}) 
                                + \sum_{k=1}^{N_\eta-1}(\gamma_{k}-\gamma_{k-1})  v_\gep^{n}(\rho_{j+k}^{n})
                                -\gamma_{N_\eta-1} v_\gep^{n}(\rho_{j+N_\eta}^{n})\\
            &=-\gamma_0
            \textcolor{blue}{(v_\gep^{n})'\Bigl(\xi_{j+\frac{1}{2}}^{n}\Bigr)\Pi(\rho_{j+1}^n,\rho_j^n)} 
            \Delta_{j+\frac{1}{2}}^{n}\\
             &\quad\;-\sum_{k=1}^{N_\eta-1}(\gamma_{k}-\gamma_{k-1})  
            \textcolor{blue}{(v_\gep^{n})'\Bigl(\xi_{j+\frac{1}{2} +k}^{n}\Bigr)\Pi(\rho_{j+1+k}^n,\rho_{j+k}^n)}
            \Delta_{j+\frac{1}{2}+k}^{n}\\
                                &\quad\;+\gamma_{N_\eta-1} 
            \textcolor{blue}{(v_\gep^{n})'\Bigl(\xi_{j+\frac{1}{2}+N_\eta}^{n}\Bigr) \Pi(\rho_{j+1+N_\eta}^n,\rho_{j+N_\eta}^n)}
            \Delta_{j+\frac{1}{2}+N_\eta}^{n}.
    \end{align*}
    Here $\xi_{j+1/2}^{n}$ is a value between $\rho_j^n$ and $\rho_{j+1}^n$ and \textcolor{blue}{$\Pi(x,y)$ as in (\ref{eq:Pi_x_y})}.
    Now, we plug (\textasteriskcentered{}) back in and receive
    \begin{align}
        \Delta_{j+\frac{1}{2}}^{n+1}&= 
            \left(1-\lambda \Bigl( V_{\gep,j+1}^{n}-\gamma_0(v_\gep^{n})'(\xi_{j+\frac{1}{2}}^{n})\rho_j^n\Bigr) 
            \right)\textcolor{blue}{\Pi(\rho_{j+1}^n,\rho_j^n)}\Delta_{j+\frac{1}{2}}^{n} \tag{i}\\
            &\quad\;+\lambda V_{\gep,j-1}^{n} \Delta_{j-\frac{1}{2}}^{n}\tag{ii}\\
            &\quad\;+\lambda \rho_j^n \sum_{k=1}^{N_\eta-1}(\gamma_{k}-\gamma_{k-1})
            (v_\gep^{n})'(\xi_{j+\frac{1}{2} +k}^{n})\textcolor{blue}{\Pi(\rho_{j+1+k}^n,\rho_{j+k}^n)}\Delta_{j+\frac{1}{2}+k}^{n}\tag{iii}\\
            &\quad\;-\lambda\rho_j^n \gamma_{N_\eta-1} (v_\gep^{n})'(\xi_{j+\frac{1}{2}+N_\eta}^{n})\textcolor{blue}{\Pi(\rho_{j+1+N_\eta}^n,\rho_{j+N_\eta}^n)}\Delta_{j+\frac{1}{2}+N_\eta}^{n}\tag{iv}.
    \end{align}
    Since
    \begin{itemize}
        \item $V_{\gep(\omega),j-1}^{n} \geq 0 \quad \forall n \in \{0,\dots,N_T\},\; \forall j \in \Z,\; \forall \omega \in \Omega$
               (Eq. (\ref{eq:V_gep_grt_0})),
        \item $\gamma_k - \gamma_{k-1} \leq 0 \quad \forall k=1,\dots,N_\eta$  
            (Rem. \ref{rem:ass_W} and Eq. (\ref{eq:gamma_k})),
        \item \textcolor{blue}{$({v_\gep^n})' (\rho) \leq 0 \quad \forall n \in \{0,\dots,N_T\}, \; \rho \in [0,\rho^{max}]$ 
        (Def. \ref{def:def_v'} and Eq. (\ref{eq:def_v^n}))}, 
        \item \textcolor{blue}{$\Pi(x,y) \in (0,1] \quad x \neq y$  (Eq. \ref{eq:Pi_x_y})},
    \end{itemize}
    the terms (ii)-(iv) before the differences $\Delta_{j+1/2+(\cdot)}^{n}$ are positive. 
    Due to the \textcolor{blue}{adapted CFL condition (\ref{eq:sCFL})}, we have
    \begin{align*}
        0
        \leq \lambda \Bigl( V_{\gep,j+1}^{n}-\gamma_0(v_\gep^{n})'(\xi_{j+\frac{1}{2}}^{n})\rho_j^n\Bigr) 
        \leq \textcolor{blue}{\lambda \Bigl( \lVert v_\gep \rVert + \gamma_0 \lVert v_\gep' \rVert \lVert \rho \rVert \Bigr)
        \leq 1}.
    \end{align*}
    Hence, the term (i) before $\Delta_{j+1/2}^{n}$ is positive as well. \textcolor{blue}{Using $|\Pi(x,y)|\leq1$} and    
    summing over $\Z$, we obtain
    \begin{align*}
        \sum_{j \in \Z} \Bigl|\Delta_{j+\frac{1}{2}}^{n+1}\Bigr| &\leq
        \sum_{j \in \Z}\left(1-\lambda \Bigl( V_{\gep,j+1}^{n}-\gamma_0(v_\gep^{n})'(\xi_{j+\frac{1}{2}}^{n})\rho_j^n\Bigr) 
            \right)\Bigl|\Delta_{j+\frac{1}{2}}^{n}\Bigr| \\
            &\quad\;+\lambda \sum_{j \in \Z} V_{\gep,j-1}^{n} \Bigl|\Delta_{j-\frac{1}{2}}^{n}\Bigr|\\
            &\quad\;+\lambda \sum_{j \in \Z} \rho_j^n \sum_{k=1}^{N_\eta-1}(\gamma_{k}-\gamma_{k-1})
            (v_\gep^{n})'(\xi_{j+\frac{1}{2} +k}^{n})\Bigl|\Delta_{j+\frac{1}{2}+k}^{n}\Bigr|\\
            &\quad\;-\lambda \sum_{j \in \Z} \rho_j^n \gamma_{N_\eta-1} (v_\gep^{n})'(\xi_{j+\frac{1}{2}+N_\eta}^{n})\Bigl|\Delta_{j+\frac{1}{2}+N_\eta}^{n}\Bigr|,
    \end{align*}
    which can be written, due to the summation on the whole space, as
    \begin{align}
        \sum_{j \in \Z} \Bigl|\Delta_{j+\frac{1}{2}}^{n+1}\Bigr|
             &\leq \sum_{j \in \Z} \Biggl[ 
                1-\lambda \Bigl(V_{\gep,j+1}^{n}-V_{\gep,j}^{n}\Bigr)\nonumber\\
                &\quad\;-(v_\gep^{n})'(\xi_{j+\frac{1}{2}}^{n}) \lambda \Biggl(
                    -\gamma_0 \rho_j^n + \sum_{k=1}^{N_\eta-1} (\gamma_{k-1}-\gamma_k) \rho_{j-k}^n 
                    +\gamma_{N_\eta-1} \rho_{j-N_\eta}^n         
                \Biggr)
                \Biggr] \Bigl|\Delta_{j+\frac{1}{2}}^{n}\Bigr|\nonumber\\
                &\leq \sum_{j \in \Z} \Biggl[
                    1-\underbrace{\lambda  \Bigl(V_{\gep,j+1}^{n}-V_{\gep,j}^{n}\Bigr)}_
                    {\mathrlap{\leq\,\lambda \gamma_0 \lVert {v_\gep^n}' \rVert (\sup_{j} \{\rho_j^0\}-\rho_j^n) \leq \,\lambda \gamma_0 \lVert v_\gep^n \rVert \lVert \rho\rVert \text{ by } (\ref{eq:V-V_bound})}}
                    +\lVert {v_\gep^n}' \rVert \lambda \gamma_0 \lVert \rho \rVert
                \Biggr] \Bigl|\Delta_{j+\frac{1}{2}}^{n}\Bigr|, \label{eq:BV_1}
    \end{align}
    where we have used a telescoping argument for the series. We further reduce the above to
    \begin{align*}
        \sum_{j \in \Z} \Bigl|\Delta_{j+\frac{1}{2}}^{n+1}\Bigr|
        &\leq \Bigl[1+\lambda \gamma_0 \bigl(\lVert {v_\gep^n} \rVert   
        +\lVert {v_\gep^n}' \rVert \bigr)  \lVert\rho\rVert \Bigr] \sum_{j \in \Z}\Bigl|\Delta_{j+\frac{1}{2}}^{n}\Bigr| \\
        &\leq \Bigl[1+\Delta t W_\eta(0) \bigl(\lVert {v_\gep} \rVert   
        +\lVert {v_\gep}' \rVert \bigr)  \lVert\rho\rVert \Bigr] \sum_{j \in \Z}\Bigl|\Delta_{j+\frac{1}{2}}^{n}\Bigr|.
    \end{align*}
    By iterative repeating of the above and using
    \begin{align} \label{eq:TV_one_dim}
     \sum_{j\in \Z} |\rho^0_{j+1}-\rho^0_{j}| = \text{TV}\left(\{\rho^0_j\}_{j \in \Z} ;\R\right) \leq 
     \text{TV}(\rho_0;\R),
     \end{align}
     we obtain
    \begin{align}
        \text{TV}(\rho^{\Delta x}(T,\cdot);\mathbb{R}) &\leq \Bigl[1+\Delta t W_\eta(0) \bigl(\lVert {v_\gep} \rVert   
        +\lVert {v_\gep}' \rVert \bigr)  \lVert\rho\rVert \Bigr]^{T/\Delta t} \text{TV}\left(\{\rho^0_j\}_{j \in \Z} ;\R\right) \nonumber\\ 
        &\leq \exp \Bigl(T \underbrace{\textcolor{blue}{W_\eta(0) (\lVert {v_\gep} \rVert   
        +\lVert {v_\gep}' \rVert )  \lVert\rho\rVert}}_{=:\, C(W_\eta,v,\rho^{max},\textcolor{blue}{\gep})}\Bigr) \text{TV}(\rho_0;\R). \qedhere
    \end{align}
\end{proof}

\begin{remark} \label{rem:BV_space_indep_e_app} 
 Similar to (\ref{eq:V-V_bound}) one can show that
    \begin{align*}
        - \Bigl(V_{\gep,j+1}^{n}-V_{\gep,j}^{n}\Bigr) 
        \leq \gamma_0 \lVert {v_\gep^n}' \rVert \Bigl(\sup_{j} \{\rho_j^n\}-\rho_j^n\Bigr)
        \leq \gamma_0 \lVert {v_\gep^n}' \rVert \lVert \rho^n \rVert
        \leq \gamma_0 \lVert v' \rVert \lVert \rho^0 \rVert,
    \end{align*}
    due to our assessment of $({v_\gep^n})'$  (Sec. \ref{subsec:ana_v_e}) and the maximum principle. This implies the bound
    \begin{align*}
        (\ref{eq:BV_1}) \leq \Bigl[1+\Delta t W_\eta(0) 2 \lVert v' \rVert  \lVert\rho\rVert \Bigr] 
            \sum_{j \in \Z}\Bigl|\Delta_{j+\frac{1}{2}}^{n}\Bigr|
    \end{align*}
    and thus
    \begin{align*}
        \text{TV}(\rho^{\Delta x}(T,\cdot);\R) 
        &\leq \exp \Bigl(T \underbrace{2 W_\eta(0) \lVert v' \rVert  \lVert\rho\rVert}_{
            =:\, \tilde{C}(W_\eta,v,\rho^{max})}\Bigr) \text{TV}(\rho_0;\R).
    \end{align*}
    
\end{remark}

Using Lemma \ref{lem:BV_space} or optionally Remark \ref{rem:BV_space_indep_e} we present the proof for Lemma \ref{lem:BV_space_time}.
\begin{proof}
    The claim will only be shown for $C$, as for $\tilde{C}$ just a few estimates have to be exchanged. 
    We adapt the proof from \cite[3.5]{Friedrich2018} and start by fixing any $T \geq 0$. 
    Now we have
    \begin{itemize}
        \item If $T \leq \Delta t$, then $\text{TV}(\rho^{\Delta x};\R \times [0,T])  \leq T \cdot \text{TV}(\rho^0;\R)$ 
                as no time step is necessary to calculate $\rho(t,x), \; \forall t\leq T$. 
        \item If $T > \Delta t$, we fix the discrete time horizon $N_T \in \N \setminus \{0\}$ such that \\
            $N_T \Delta t < T \leq (N_T+1)\Delta t$. 
            Hence, by Definition of the total variation in two dimensions and twice applying the first equality of (\ref{eq:TV_one_dim}), the total variation can be written as 
            \begin{align}
                \text{TV}(\rho^{\Delta x};\R \times [0,T])&=
                \underbrace{\sum_{n=0}^{N_T-1} \sum_{j \in \Z} \Delta t |\rho_{j+1}^{n}-\rho_{j}^{n}|+(T-N_T\Delta t) \sum_{j \in \Z} |\rho_{j+1}^{N_T}-\rho_{j}^{N_T}|}
                _{\mathclap{\leq T \cdot \left( \exp (T C ) \text{TV}(\rho_0;\R) \right)  }} \nonumber\\
                &\quad\;+\sum_{n=0}^{N_T-1} \sum_{j \in \Z} \Delta x |\rho_{j}^{n+1}-\rho_{j}^{n}|. \label{eq:TV_proof1}
            \end{align}
            Therefore, we are left with bounding the last term. Considering our scheme, we derive
            \begin{align*}
                \rho_j^{n+1}-\rho_j^n &= \lambda \Bigl(V_{\gep,j-1}^{n} \rho_{j-1}^{n}-V_{\gep,j}^{n} \rho_{j}^{n}\Bigr) \\
                \overset{\pm 0}&{=}\lambda \left(  \Bigl(V_{\gep,j-1}^{n}-V_{\gep,j}^{n} \Bigr) \rho_{j-1}^{n}-
                                                    V_{\gep,j}^{n}\Bigl( \rho_{j}^{n}-\rho_{j-1}^{n}\Bigr)   \right)\\
                &= \lambda \Bigl(-\rho_{j-1}^{n} \Bigl( \sum_{k=0}^{N_\eta-1} \gamma_k \bigl(v_\gep^n(\rho_{j+k+1}^n)-v_\gep^n(\rho_{j+k}^n)\bigr) \Bigr)
                -V_{\gep,j}^n(\rho_{j}^n-\rho_{j-1}^n)  \Big)\\
                &= \lambda \Bigl(-\rho_{j-1}^{n} \Bigl( \sum_{k=0}^{N_\eta-1} \gamma_k 
                \textcolor{blue}{(v_\gep^n)'\Pi(\rho_{j+k+1}^n,\rho_{j+k}^n)(\rho_{j+k+1}^n-\rho_{j+k}^n)} \Bigr)
                -V_{\gep,j}^n(\rho_{j}^n-\rho_{j-1}^n)  \Big)
            \end{align*}
            where we added and subtracted $V_{\gep,j}^n \rho_{j-1}^n$, applied our scheme and lastly 
            \textcolor{blue}{used our assessment of $v_\gep'$ as in Section \ref{subsec:ana_v_e}}. Taking absolute values, 
            \textcolor{blue}{applying
            $\lVert {v_\gep^n}' \rVert \leq \lVert {v_\gep}' \rVert \leq \lVert {v}' \rVert$, $|V_{\gep,j}^n|\leq \lVert v_\gep^n \rVert$ and $|\Pi(x,y)| \leq 1$} yields
            \begin{align*}
                |\rho_j^{n+1}-\rho_j^n| &\leq \lambda \Biggl( \lVert\rho \rVert \lVert {v}' \rVert \sum_{k=0}^{N_\eta-1} \gamma_k |\rho_{j+k+1}^n-\rho_{j+k}^n| 
                + \textcolor{blue}{\lVert v_\gep^n \rVert} |\rho_{j}^n-\rho_{j-1}^n| \Biggr). 
            \end{align*}
            Summing over $\Z$ and \textcolor{blue}{applying (\ref{eq:Dv_e_norm_approx_final}) again}, gives us
            \begin{align}\label{eq:BV_space_time_eq}
                \sum_{j \in \Z} \Delta x |\rho_j^{n+1}-\rho_j^n| &\leq \underbrace{\Delta x \lambda}_{\Delta t} \sum_{j \in \Z} |\rho_j^{n}-\rho_{j-1}^n|
                \Bigl( W_\eta(0) \lVert v' \rVert \lVert\rho \rVert +\textcolor{blue}{\lVert v_\gep \rVert}\Bigr).
            \end{align}
            Finally, summing over the time steps
            \begin{align}\label{eq:K_1_gep_T}
                \sum_{n=0}^{N_T-1} \sum_{j \in \Z} \Delta x |\rho_j^{n+1}-\rho_j^n| 
                &\leq T \cdot \Bigl[ \exp \bigl(T C(W_\eta,v,\rho^{max},\textcolor{blue}{\gep})\bigr) \text{TV}(\rho_0;\R)\nonumber\\
                & \quad \qquad \cdot \bigl(W_\eta(0) \lVert v' \rVert \lVert \rho \rVert + \textcolor{blue}{\lVert v_\gep \rVert} \bigr)\Bigr] \nonumber\\
                & =: T \cdot \textcolor{blue}{\mathcal{K}_\gep(T)}, 
            \end{align}
            applying our knowledge on the spatial TV bound (Lem. \ref{lem:BV_space}). Altogether with (\ref{eq:TV_proof1}) we obtain
            \begin{align*}
                \text{TV}(\rho^{\Delta x};\R \times [0,T]) \leq T \exp \Bigl(T C(W_\eta,v,\rho^{max},\textcolor{blue}{\gep}) \Bigr) 
                \Bigl( 1+W_\eta(0) \lVert v' \rVert \lVert \rho \rVert + \textcolor{blue}{\lVert v_\gep \rVert}\Bigr)
                \text{TV}(\rho_0;\R)
            \end{align*}
            which is the claim. \qedhere
    \end{itemize}
\end{proof} 

\begin{remark}
    As in \cite[2.24]{Friedrich2021_diss}, from (\ref{eq:BV_space_time_eq}) it follows, analogous to (\ref{eq:K_1_gep_T}) that
    \begin{align}\label{eq:Delta-t}
        \sum_{j \in \Z} \Delta x |\rho_j^{n+1}-\rho_j^n| &\leq \Delta t \cdot \textcolor{blue}{\mathcal{K}_\gep(T)},
    \end{align}
    which will become of use in Remark \ref{rem:time_cts_rho}.\\
\end{remark}

\subsubsection{Discrete entropy inequality} \label{lem:discrete_enropy_ineq_app}
For the proof of Lemma \ref{lem:discrete_enropy_ineq}, we prove
 that under our assumptions on the error term, $v_\epsilon$
    is an admissible velocity function, although it neither satisfies the classical assumptions of 
    \cite[2.25]{Friedrich2021_diss} nor the NV model (Rem. \ref{rem:ass_v_ND}), whilst reducing the proof from a 1-to-1 case
    to our setting of a singular road. Additionally, we use the fact that $V_{\epsilon,j}^n \geq 0$.
    
\begin{proof} 
    \textcolor{blue}{First note that, as pointed out in Section \ref{sec:num schemes nonloc_stoch}, the numerical flux is a random variable in every time step,
    thus actually $F_{j+1/2}^n= F_{j+1/2}^{n}\bigl( \epsilon(t^n,\omega) \bigr) ,\;j \in \Z, \, \omega \in \Omega$.
    Yet, spatial differentiability is still given since $\epsilon(t^n,\omega)$ is constant in the spatial dimension. Therefore, we
    omit the notation with respect to $\omega$ here as well.} 
    Let 
    \begin{align*}
        G^n_j(u,w):=w-\lambda \left( F_{j+1/2}^n(w)-F_{j-1/2}^n(u) \right).
    \end{align*}
    Then $G_j$ 
    is monotone with respect to both its arguments as
    \begin{align*}
        \frac{\partial G_j^n}{\partial w}=1-\lambda  V_{\gep,j}^n \geq 0,
        \quad \frac{\partial G_j^n}{\partial u} =\lambda  V_{\gep,j-1}^n \geq 0, 
    \end{align*}
    \textcolor{blue}{due to the CFL condition (\ref{eq:sCFL}), the assumptions on $\gep$ (\ref{eq:construction_eps}) and the non negativity of $v_\gep$}.\\
    The monotonicity implies that 
    \begin{align}
        G_j^n (\rho_{j-1}^n \wedge c, \rho_{j}^n \wedge c)  &\geq  G_j^n (\rho_{j-1}^n,\rho_{j}^n) \wedge  G_j^n (c,c), \label{eq:a1}\tag{a} \\
        G_j^n (\rho_{j-1}^n \vee c, \rho_{j}^n \vee  c)  &\leq  G_j^n (\rho_{j-1}^n,\rho_{j}^n) \vee  G_j^n (c,c). \label{eq:b1}\tag{b}
    \end{align}
    Subtracting (\ref{eq:b1}) from (\ref{eq:a1}), we obtain
    \begin{align}\label{a2}
        \left\lvert G_j^{n}(\rho_{j-1}^n,\rho_{j}^n)-G_j^n(c,c) \right\rvert  &\leq G_j^n (\rho_{j-1}^n \wedge c, \rho_{j}^n \wedge c) - G_j^n (\rho_{j-1}^n \vee c, \rho_{j}^n \vee c) \nonumber \\
        &=|\rho_j^n-c|-\lambda \left( H_{j+1/2}^n(\rho_j^n) - H_{j-1/2}^n(\rho_{j-1}^n) \right). 
    \end{align}
    Now we estimate the left side of (\ref{a2}) via
    \begin{align}\label{b2}
        \left\lvert G_j^{n}(\rho_{j-1}^n,\rho_{j}^n)-G_j^n(c,c) \right\rvert &=
        \left\lvert \rho_j^{n+1}-c+\lambda \left( F_{j+1/2}^n(c)-F_{j-1/2}^n(c) \right) \right\rvert \nonumber \\
        &\geq \text{ sign}(\rho_j^{n+1}-c)\left(\rho_j^{n+1}-c+\lambda \left( F_{j+1/2}^n(c)-F_{j-1/2}^n(c) \right)\right) \nonumber \\
        &=\left\lvert \rho_j^{n+1}-c \right\rvert + \lambda \text{ sign}(\rho_j^{n+1}-c) \left( F_{j+1/2}^n(c)-F_{j-1/2}^n(c) \right),
    \end{align}
    where we have used the definition of $G_j^{n}$ and the properties of sign w.r.t. $\lvert \cdot \rvert$. 
    Lastly, combining (\ref{a2}) with (\ref{b2}) and subtracting $\left[  \lambda \text{ sign}(\rho_j^{n+1}-c) \left( F_{j+1/2}^n(c)-F_{j-1/2}^n(c) \right) \right] $
    on both sides yields the desired claim.
\end{proof}
\pagebreak 

\subsubsection{Convergence}\label{lem:convergence_app} 
The following proof concerns Lemma \ref{lem:convergence}.

\begin{proof}

    Let $\phi \in C_0^1([0,T)\times \R;\R^+)$ and set $\phi_j^n=\phi(t^n,x_j)$ $\forall j \in \Z, \, n\in \N$.
    We start by multiplying the discrete entropy inequality (\ref{eq:discrete_enrtropy_ineq}) by $\phi^n_j \Delta x$ and
    sum over space and time, i.e. $n \in \{0,\dots,N_T-1\}$ and $j \in \Z$, using $\lambda=\Delta t / \Delta x$:
    \begin{align*}
        \Delta x \Delta t \sum_{n=0}^{N_T-1} \sum_{j \in Z} 
        &- \Bigl(|\rho_j^{n+1}-c|-|\rho_j^n-c|\Bigr)  \phi_j^n/\Delta t \\
        &- \Bigl(H_{j+1/2}^n(\rho_j^n)-H_{j-1/2}^n(\rho_{j-1}^n)\Bigr) \phi_j^n/\Delta x\\
        &- \text{ sign}(\rho_j^{n+1}-c) \Bigl(F_{j+1/2}^n(c)-F_{j-1/2}^n(c)\Bigr) \phi_j^n/\Delta x\ \geq 0.
    \end{align*}
    Our goal is to show that the limit of this expression satisfies the entropy
     inequality of Definition \ref{def:nonlocal_weak_entropy_sol}.
    To achieve this, we will commence by breaking the above down into its main components using summation by parts 
    to obtain the partial derivatives $\partial_x \phi,\, \partial_t \phi$.
    We obtain
    \begin{align}
        &-\Delta x \sum_{j \in \Z} | \rho_j^{N_T} -c |\phi_j^{N_T} \label{eq:1}\\
        &+\Delta x \Delta t \sum_{n=0}^{N_T-1} \sum_{j \in Z} |\rho_j^{n+1}-c| (\phi_j^{n+1}-\phi_j^{n})/\Delta t 
        +\Delta x \sum_{j \in Z} |\rho_j^0-c| \phi_j^0       \label{eq:2} \\
        &+\Delta x \Delta t \sum_{n=0}^{N_T-1} \sum_{j \in Z} H_{j-1/2}^n(\rho_{j-1}^n) (\phi_j^n-\phi_{j-1}^n)/\Delta x \label{eq:3}\\
        &-\Delta x \Delta t \sum_{n=0}^{N_T-1} \sum_{j \in Z} \text{ sign}(\rho_j^{n+1}-c) \Bigl(F_{j+1/2}^n(c)-F_{j-1/2}^n(c)\Bigr) \phi_j^n/\Delta x\ \geq 0. \label{eq:4}
    \end{align}
    We now analyze every term individually.
    Since $\phi \in C_0^1([0,T)\times \R;\R^+)$ and $\rho^{\Delta x} \rightarrow \rho,\, \Delta x \rightarrow 0$ in $L^1_{loc}$ by assumption,
    it follows  that
       \begin{align*}
            (\ref{eq:1}) \rightarrow \, -  \int_{\R} |\rho_T(x)-c| \phi(T,x) \,dx=0,\quad \Delta x \rightarrow 0.
        \end{align*}\\
    As per assumption $\rho^{\Delta x} \rightarrow \rho,\, \Delta x \rightarrow 0$ in $L^1_{loc}$ and 
        $\dfrac{\phi_j^{n+1}-\phi_j^{n}}{\Delta t} \rightarrow \partial_t \phi, \, \Delta t \rightarrow 0$, we also obtain
        \begin{align*}
            (\ref{eq:2}) \rightarrow  \, \int_{0}^{T} \int_{\R} |\rho-c| \partial_t \phi \,dx  \,dt 
                + \int_{\R} |\rho_0(x)-c| \phi(0,x) \,dx,\quad \Delta x \rightarrow 0.
        \end{align*} \\
    Next, by definition of $H_{j-1/2}^n$ we have
        \begin{align*}
            H_{j-1/2}^n(\rho_{j-1}^n) 
            &= F_{j-1/2}^n(\rho_{j-1}^n \wedge c)-F_{j-1/2}^n(\rho_{j-1}^n \vee c) \\
            &= \text{sign}(\rho_{j-1}^n-c) \Bigl(F_{j-1/2}^n(\rho_{j-1}^n)-F_{j-1/2}^n(c)\Bigr).
        \end{align*}
        Further, as
        $\dfrac{\phi_j^{n}-\phi_{j-1}^{n}}{\Delta x} \rightarrow \partial_x \phi, \, \Delta x \rightarrow 0$
        and per assumption
        \begin{align*}
            F_{j-1/2}^n(\rho_{j-1}^n)  \rightarrow f_{\textcolor{blue}{\gep}}(t,x,\rho), \quad \Delta x \rightarrow 0 \text{ in } L^1_{loc},
        \end{align*}
        we can also conclude
        \begin{align*}
            (\ref{eq:3}) \rightarrow \int_{0}^{T} \int_{\R} \text{sign}(\rho-c) 
            \Bigl(f_{\textcolor{blue}{\gep}}(t,x,\rho)-f_{\textcolor{blue}{\gep}}(t,x,c)\Bigr) \,dx  \,dt,
            \quad \Delta x \rightarrow 0,
        \end{align*}
        \textcolor{blue}{due to our construction of a well-posed time integrable random process $\gep(t,\omega)$}. \\
        However, a more detailed evaluation is required for the last term. 
        We apply the definition of the numerical flux, introduce a zero term and then obtain
        \begin{align}
            (\ref{eq:4})
            \overset{\text{def.}}&{=}&& -\Delta x \Delta t \sum_{n=0}^{N_T-1} \sum_{j \in Z} \text{sign}(\rho_j^{n+1}-c) 
            \left(c \frac{V_{\gep,j}^n-V_{\gep,j-1}^n}{\Delta x}   \right) \phi_j^n \nonumber \\
            \overset{\pm 0}&{=}&& -\Delta x \Delta t \sum_{n=0}^{N_T-1} \sum_{j \in Z} 
            \Bigl(\text{sign}(\rho_j^{n+1}-c)-\text{sign}(\rho_j^{n}-c)\Bigr)
            \left(c \frac{V_{\gep,j}^n-V_{\gep,j-1}^n}{\Delta x}   \right) \phi_j^n   \label{eq:5} \\   
            &&& -\Delta x \Delta t \sum_{n=0}^{N_T-1} \sum_{j \in Z} 
            \text{sign}(\rho_j^{n}-c)
            \left(c \frac{V_{\gep,j}^n-V_{\gep,j-1}^n}{\Delta x}   \right) \phi_j^n. \label{eq:6}
        \end{align}
        Due to (\ref{eq:bound_on_diff_deltax}) it holds
         $| V_{\gep,j}^{n}- V_{\gep,j-1}^{n}| \leq \Delta x  W_\eta(0)  \eta \textcolor{blue}{\lVert v_\gep \rVert} \lVert \rho \rVert 
         =\mathcal{O}(\Delta x)$, and we can bound the finite differences to obtain for the second term
         \begin{align*}
            (\ref{eq:6}) \rightarrow - \int_{0}^{T} \int_{\R} \text{sign}(\rho-c) (c \partial_x V_\gep) \phi \,dx \,dt, 
             \quad \Delta x \rightarrow 0.
         \end{align*}
        It remains to show, that the first term vanishes.
        To achieve this, we perform summation by parts,
        adding and subtracting $(V_{\gep,j}^{n+1}-V_{\gep,j-1}^{n+1})\phi_j^n$ in the process 
        and obtain
        \begin{align}\label{eq:7}
            (\ref{eq:5})= \,
            &\Delta t \sum_{n=0}^{N_T-2} \sum_{j \in Z} \text{sign}(\rho_j^{n+1}-c) \phi_j^n c
            \Bigl[(V_{\gep,j}^{n+1}- V_{\gep,j-1}^{n+1})-(V_{\gep,j}^{n}- V_{\gep,j-1}^{n})\Bigr]\\
            +&\Delta t \Delta t \Delta x \sum_{n=0}^{N_T-2} \sum_{j \in Z} \text{sign}(\rho_j^{n+1}-c) c 
                \frac{V_{\gep,j}^{n+1}- V_{\gep,j-1}^{n+1}}{\Delta x} \frac{\phi_{j}^{n+1}-\phi_{j}^{n}}{\Delta t}\nonumber  \\
            +&\Delta t \Delta x \sum_{j \in Z} \text{sign}(\rho_j^{0}-c) \phi_j^0 c 
                \frac{V_{\gep,j}^{0}- V_{\gep,j-1}^{0}}{\Delta x}\nonumber\\
            -&\Delta t \Delta x \sum_{j \in Z} \text{sign}(\rho_j^{N_T-1}-c) \phi_j^{N_T-1} c 
                \frac{V_{\gep,j}^{N_T-1}- V_{\gep,j-1}^{N_T-1}}{\Delta x}.\nonumber
        \end{align}
        The occurring differences can all be bounded by the same argument as above for $| V_{\gep,j}^{n}- V_{\gep,j-1}^{n}| $ and therefore 
        the last three terms vanish as $\Delta x \rightarrow 0$ and $\Delta t \rightarrow 0$ respectively,
        further using the compactness of the test functions as e.g. in (\ref{eq:1}).
        For the first term we derive as in (\ref{eq:max_p_V_eq})
        \begin{align}\label{eq:8}
            (V_{\gep,j}^{n+1}- V_{\gep,j-1}^{n+1})-(V_{\gep,j}^{n}- V_{\gep,j-1}^{n})
            = \sum_{k=1}^{N_\eta-1}
            &(\gamma_{k-1}-\gamma_{k}) 
                &&\Bigl(v_\gep^{\textcolor{blue}{n+1}}  (\rho_{j+k}^{n+1})-v_\gep^{\textcolor{blue}{n}}(\rho_{j+k}^{n}) \Bigr)\nonumber\\
            +&\;\gamma_{N_\eta-1}
                &&\Bigl(v_\gep^{\textcolor{blue}{n+1}}  (\rho_{j+N_\eta}^{n+1})-v_\gep^{\textcolor{blue}{n}}(\rho_{j+N_\eta}^{n}) \Bigr)\nonumber\\
            -&\;\gamma_{0}
                &&\Bigl(v_\gep^{\textcolor{blue}{n+1}} (\rho_{j}^{n+1})-v_\gep^{\textcolor{blue}{n}}(\rho_{j}^{n}) \Bigr).
        \end{align}

        We once again use the compact support of $\phi$ in space and time, and notice, that
        there must exist an $R$, such that $\phi(t,x)=0, \; \forall t>0$ and $|x|>R$. For the discrete variant we
        choose the indices $j_0,j_1 \in \Z$, such that $-R \in (x_{j_0-1/2},x_{j_0+1/2}]$ and $R \in (x_{j_1-1/2},x_{j_1+1/2}]$.
        Then it follows that $\sum_{j \in \Z} \phi_j^n = \sum_{j=j_0}^{j_1}\phi_j^n$.\\

        With the above, we can now plug (\ref{eq:8}) in (\ref{eq:7}), multiply by $\Delta x / \Delta x$ and use
        the obvious bound on $\phi$ to obtain
        \begin{align*}
            (\ref{eq:7}) \leq \frac{\Delta t}{\Delta x} {\left\lVert \phi \right\rVert}_{L^\infty([0,T]\times[-R,R])}c
            \Biggl|
                \sum_{k=1}^{N_\eta-1} 
                &(\gamma_{k-1}-\gamma_{k}) &&\Delta x \sum_{n=0}^{N_T-2} \sum_{j=j_0}^{j_1}
                    v_\gep^{\textcolor{blue}{n+1}}(\rho_{j+k}^{n+1})-v_\gep^{\textcolor{blue}{n}}(\rho_{j+k}^{n}) \\
                +&\;\gamma_{N_\eta-1} &&\Delta x \sum_{n=0}^{N_T-2} \sum_{j=j_0}^{j_1}
                    v_\gep^{\textcolor{blue}{n+1}}(\rho_{j+N_\eta}^{n+1})-v_\gep^{\textcolor{blue}{n}}(\rho_{j+N_\eta}^{n}) \\
                -&\;\gamma_{0} &&\Delta x \sum_{n=0}^{N_T-2} \sum_{j=j_0}^{j_1}
                    v_\gep^{\textcolor{blue}{n+1}}(\rho_{j}^{n+1})-v_\gep^{\textcolor{blue}{n}}(\rho_{j}^{n}) \;
            \Biggr|
            .
        \end{align*}
        
        \textcolor{blue}{
        A major difference to \cite[2.2.6]{Friedrich2021_diss} is the appearance of the noise induced time dependency of the velocity function, i.e. $v_\gep^{n+1}$, $v_\gep^{n}$. 
         Therefore, we cannot simply apply the mean value theorem, but have bound to  the  differences of the noise. By the Lipschitz-continuity of $x \mapsto x^+$ it holds that
        \begin{align*}
            v_\gep^{n+1}(\rho_{j}^{n+1})-v_\gep^{n}(\rho_{j}^{n}) &= 
            \bigl(v(\rho_{j}^{n+1})-\gep(t_{n+1})\bigr)^+-\bigl(v(\rho_{j}^{n})-\gep(t_{n})\bigr)^+ \\
            &=\bigl(v(\rho_{j}^{n+1})-\gep(t_{n})+\gep(t_{n})-\gep(t_{n+1})\bigr)^+-\bigl(v(\rho_{j}^{n})-\gep(t_{n})\bigr)^+ \\
            &\leq \lVert v_\gep' 
            \rVert \bigl(|\rho_{j}^{n+1}-\rho_{j}^{n})|+|\gep(t_{n+1})-\gep(t_{n})|\bigr).
        \end{align*}
        }
        Hence, it follows by $\gamma_{k-1}-\gamma_{k}\geq0$ and the triangle inequality  
        \begin{align*}
            (\ref{eq:7}) &\leq \frac{\Delta t}{\Delta x} {\left\lVert \phi \right\rVert}_{L^\infty([0,T]\times[-R,R])}c \textcolor{blue}{\left\lVert v_\gep'\right\rVert}  \\
            &\quad\quad\quad\quad\quad\quad \cdot\Biggl(
                \sum_{k=1}^{N_\eta-1} 
                (\gamma_{k-1}-\gamma_{k}) &&\Delta x \sum_{n=0}^{N_T-2} \sum_{j=j_0}^{j_1}
                    |\rho_{j+k}^{n+1}-\rho_{j+k}^{n}|+ \textcolor{blue}{|\gep(t_{n+1})-\gep(t_{n})|}\\
                &\quad\quad\quad\quad\quad\quad\quad\quad \;\;+\gamma_{N_\eta-1} &&\Delta x \sum_{n=0}^{N_T-2} \sum_{j=j_0}^{j_1}
                     |\rho_{j+N_\eta}^{n+1}-\rho_{j+N_\eta}^{n}|+ \textcolor{blue}{|\gep(t_{n+1})-\gep(t_{n})|}\\
                &\quad\quad\quad\quad\quad\quad\quad\quad \;\;+\gamma_{0} &&\Delta x \sum_{n=0}^{N_T-2} \sum_{j=j_0}^{j_1}
                    |\rho_{j}^{n+1}-\rho_{j}^{n}|+ \textcolor{blue}{|\gep(t_{n+1})-\gep(t_{n})|}                    \; \Biggr).
        \end{align*}        
        The terms of the form 
        $\Delta x \sum_{n=0}^{N_T-2} \sum_{j=j_0}^{j_1}|\rho_{j}^{n+1}-\rho_{j}^{n}|$ 
        can be bounded as in (\ref{eq:K_1_gep_T}), such that  
        \begin{align*}
            (\ref{eq:7})
            &\leq \frac{\Delta t}{\Delta x} {\left\lVert \phi \right\rVert}_{L^\infty([0,T]\times[-R,R])}c\textcolor{blue}{\left\lVert v_\gep'\right\rVert}
            \Bigl(2 \gamma_0 T \textcolor{blue}{\mathcal{K}_\gep(T)}   \\
            &\quad\quad\quad\quad\quad\quad\quad\quad\quad\quad\quad\quad\quad 
            \textcolor{blue}{+\sum_{k=1}^{N_\eta-1} 
                (\gamma_{k-1}-\gamma_{k})} && \textcolor{blue}{\Delta x \sum_{n=0}^{N_T-2} \sum_{j=j_0}^{j_1}
                     |\gep(t_{n+1})-\gep(t_{n})|} \\
                &\quad\quad\quad\quad\quad\quad\quad\quad\quad\quad\quad\quad\quad 
                \textcolor{blue}{+\gamma_{N_\eta-1}} && \textcolor{blue}{\Delta x \sum_{n=0}^{N_T-2} \sum_{j=j_0}^{j_1}
                     |\gep(t_{n+1})-\gep(t_{n})|} \\
                &\quad\quad\quad\quad\quad\quad\quad\quad\quad\quad\quad\quad\quad 
                \textcolor{blue}{+\gamma_{0}} && \textcolor{blue}{\Delta x \sum_{n=0}^{N_T-2} \sum_{j=j_0}^{j_1}
                    |\gep(t_{n+1})-\gep(t_{n})|}  \Bigr).
        \end{align*}
        \textcolor{blue}{Since $\gep \in \text{BV}$ we may bound the differences by the respective norm, which gives}
        \begin{align*}
        \textcolor{blue}{
        \Delta x \sum_{n=0}^{N_T-2} \sum_{j=j_0}^{j_1}
                    |\gep(t_{n+1})-\gep(t_{n})| \leq         
                    \Delta x \sum_{j=j_0}^{j_1}  \lVert \gep \rVert_{\text{BV}} \leq
                   \lVert \gep \rVert_{\text{BV}} \,2 R.}
        \end{align*}
        Thus, we can conclude
        \begin{align*}
            (\ref{eq:7}) 
            &\leq \frac{\Delta t}{\Delta x} {\left\lVert \phi \right\rVert}_{L^\infty([0,T]\times[-R,R])}c\textcolor{blue}{\left\lVert v_\gep'\right\rVert}
            \Bigl(2 \gamma_0 \bigl( T \textcolor{blue}{\mathcal{K}_\gep(T)\textcolor{blue}{+ \lVert \gep \rVert_{\text{BV}} \,2 R}} \bigr) 
            \Bigr)\\
            & \leq   \lambda \Delta x{\left\lVert \phi \right\rVert}_{L^\infty([0,T]\times[-R,R])}c\textcolor{blue}{\left\lVert v_\gep'\right\rVert}
            \Bigl(2 W_\eta(0) \bigl(  T \textcolor{blue}{\mathcal{K}_\gep(T) \textcolor{blue}{+\lVert \gep \rVert_{\text{BV}} \,2 R}} \bigr)
            \Bigr),
        \end{align*}
        using $\lambda=\Delta t /\Delta x$ and $\gamma_0 \leq W_\eta(0) /\Delta x$. This estimate converges to zero as $\Delta x \rightarrow 0$,
        hence finishing the proof.
\end{proof}

\subsubsection{Uniqueness }\label{thm:uniqe_sNV_app} 
We now present the detailed proof for Theorem \ref{thm:uniqe_sNV}.
\begin{proof} 
    The main work here is to show that $v_\gep$ is a valid substitute to $v$ as in the deterministic model (\ref{eq:NV}). 
    As $\rho,\sigma$ are weak entropy solutions of
    \begin{align*}
        &\partial_t \rho(t,x) + \partial_x \bigl(\rho(t,x) V_\gep(t,x)\bigr)=0,  && V_\gep=W_\eta*\textcolor{blue}{v_\gep}(\rho,\textcolor{blue}{t}),  &\rho(0,x)=\rho_0(x),  \\
        &\partial_t \sigma(t,x) + \partial_x \bigl(\sigma(t,x) U_\gep(t,x)\bigr)=0,  && U_\gep=W_\eta*\textcolor{blue}{v_\gep}(\sigma,\textcolor{blue}{t}),  &\sigma(0,x)=\sigma_0(x),  
    \end{align*}
    the following holds
    \begin{enumerate}
        \item $0 \leq V_\gep,\, U_\gep \leq v^{max} \textcolor{blue}{+\tau}$, since \\
         $0 < W_\eta,\, W_0=1$ and \textcolor{blue}{$0 \leq v_\gep \leq v^{max}+\tau$ by construction}.
        \item $\left\lvert \partial_x V_\gep(t,x)\right\rvert , \left\lvert \partial_x U_\gep(t,x)\right\rvert < \infty$, since\\
        \begin{align}\label{eq:2.3.1}
            \left| \partial_x V_\gep(t,x) \right| & =  \bigg| \int_{x}^{x+\eta} W_\eta'(y-x) v_\gep(\rho(t,y),t)  \,dy \nonumber \\
            & + W_\eta(\eta) v_\gep(\rho(t,x+\eta),t) - W_\eta(0) v_\gep(\rho(t,x),t)  \bigg| \\
            & \leq \lVert W_\eta' \rVert  \lVert v_\gep \rVert \eta+ 2  \lVert v_\gep \rVert W_\eta(0) \nonumber\\
            & \leq \underbrace{\lVert W_\eta' \rVert  \left(v^{max}\textcolor{blue}{+\tau}\right)\eta + 2  
            \left(v^{max}\textcolor{blue}{+\tau}\right) W_\eta(0)}_{=:\mathcal{V}_{\textcolor{blue}{\gep}}^{\infty} } < \infty. \label{eq:V_cal}
        \end{align}  
        Here we used the Leibniz-rule, the triangle inequality
        as well as our  \textcolor{blue}{bound on $v_\gep$ (\ref{eq:v_e_norm_approx_final})}. The same holds for $\partial_x U_\gep$.
        \item $V_\gep$ and $U_\gep$ are Lipschitz continuous with respect to $x$, since 
        as of 2. we have bounded first derivatives and by the definition of nonlocal weak entropy solutions it follows
        $\rho,\sigma \in C([0,T];L^1(\R))$ with $\rho(t,\cdot),\sigma(t,\cdot) \in \text{BV}(\R;\R)$.
    \end{enumerate}

    By construction and considerations 1.-3. $V_\gep$ and $U_\gep$ satisfy the assumptions of Kru\v{z}kov
    \cite{Kruzkov1970}, allowing us 
    to apply the doubling of variables technique. Thus, as in \cite{Friedrich2018} we 
    obtain
    \begin{align}\label{eq:2.3.2}
        {\left\lVert \rho(t,\cdot)-\sigma(t,\cdot) \right\rVert}_{L^1(\R)} \leq  {\left\lVert \rho_0-\sigma_0 \right\rVert}_{L^1(\R)} 
        &+ \int_{0}^{T} \int_{\R}  |\partial_x \rho(t,x)| \left\lvert U_\gep(t,x)-V_\gep(t,x)\right\rvert   \,dx  \,dt \nonumber \\
        &+ \int_{0}^{T} \int_{\R}  |\rho(t,x)| \left\lvert \partial_x U_\gep(t,x)- \partial_x V_\gep(t,x)\right\rvert   \,dx  \,dt , 
    \end{align}
    where $\partial_x \rho$ has to be understood in the sense of distributions.
    Next, we 
    \textcolor{blue}{invoke the Lipschitz-continuity}, which gives
    \begin{align}\label{eq:2.3.3}
        \left\lvert U_\gep(t,x)-V_\gep(t,x)\right\rvert &=
        \int_{x}^{x+\eta} W_\eta(y-x) \Bigl(v_\gep(\rho(t,y),t)-v_\gep(\sigma(t,y),t)\Bigr)  \,dy \nonumber \\
        & \leq W_\eta(0) \textcolor{blue}{\left\lVert v_\gep' \right\rVert} {\left\lVert \rho(t,\cdot)-\sigma(t,\cdot) \right\rVert}_{L^1(\R)}.
    \end{align}
    \textcolor{blue}{Note that the norm for $\left\lVert v_\gep' \right\rVert$ had to be defined and bounded for two dimensions, 
    as we did in Definition \ref{def:norm} and equation (\ref{eq:Dv_e_norm_approx_final})}.  For the derivatives we obtain analogous to (\ref{eq:2.3.1})
    \begin{align}\label{eq:2.3.4}
        \left\lvert \partial_x U_\gep(t,x)- \partial_x V_\gep(t,x)\right\rvert 
        & \leq \lVert W_\eta' \rVert \textcolor{blue}{\left\lVert v_\gep' \right\rVert} {\left\lVert \rho(t,\cdot)-\sigma(t,\cdot) \right\rVert}_{L^1(\R)} \nonumber \\
        & + W_\eta(0) \textcolor{blue}{\left\lVert v_\gep' \right\rVert} \Bigl(|\rho-\sigma|(t,x+\eta)+|\rho-\sigma|(t,x)\Bigr),
    \end{align}
    where we once again used the \textcolor{blue}{Lipschitz-continuity of $v_\gep$} and by Remark \ref{rem:ass_W}: $W_\eta(\eta) \leq W_\eta(0)$. 
    Next we plug our bounds (\ref{eq:2.3.3}) and (\ref{eq:2.3.4}) into (\ref{eq:2.3.2}) and obtain
    \begin{align*}
        &{\left\lVert \rho(t,\cdot)-\sigma(t,\cdot) \right\rVert}_{L^1(\R)} \\ &\quad\quad\quad\leq  {\left\lVert \rho_0-\sigma_0 \right\rVert}_{L^1(\R)} 
        &&+W_\eta(0) \lVert v_\gep' \rVert \int_{0}^{T}  {\left\lVert \rho(t,\cdot)-\sigma(t,\cdot) \right\rVert}_{L^1(\R)} \int_{\R} |\partial_x \rho(t,x)|  \,dx  \,dt \\ \displaybreak
        &&&+\lVert W_\eta' \rVert \lVert v_\gep' \rVert \int_{0}^{T}  {\left\lVert \rho(t,\cdot)-\sigma(t,\cdot) \right\rVert}_{L^1(\R)} \int_{\R} |\rho(t,x)|  \,dx  \,dt \\
        &&&+W_\eta(0) \lVert v_\gep' \rVert \int_{0}^{T} \int_{\R} \Bigl( |\rho-\sigma|(t,x+\eta)+|\rho-\sigma|(t,x) \Bigr) |\rho(t,x)|  \,dx  \,dt\\
        &\quad\quad\quad \leq {\left\lVert \rho_0-\sigma_0 \right\rVert}_{L^1(\R)}
        &&+\lVert v_\gep' \rVert \Biggl(  \Bigl( W_\eta(0) \underbrace{\sup_{t \in [0,T]} \int_{\R} |\partial_x \rho(t,x)| \, dx}_{(a)} 
        +\lVert W_\eta' \rVert \underbrace{\sup_{t \in [0,T]}  \int_{\R} | \rho(t,x)| \, dx}_{(b)}\Bigr)   \\
        &&&\quad \quad \quad \quad \cdot \int_{0}^{T} {\left\lVert \rho(t,\cdot)-\sigma(t,\cdot) \right\rVert}_{L^1(\R)} \,dt  \\
        &&&+W_\eta(0) \underbrace{\int_{0}^{T} \int_{\R} \Bigl(|\rho-\sigma|(t,x+\eta)+|\rho-\sigma|(t,x)\Bigr) \,dx  \,dt}_{(c)} \\
        &&&\quad \quad \quad \quad \cdot  \underbrace{\sup_{t \in [0,T]}  \int_{\R} | \rho(t,x)| \, dx}_{(b)}
         \Biggr).
    \end{align*}
    Now
    \begin{align*}
        & (a) \leq \sup_{t \in [0,T]} {\left\lVert \rho(t,\cdot)\right\rVert}_{\text{BV}(\R)} 
        \text{, with } {\left\lVert \rho(t,\cdot)\right\rVert }_{\text{BV}(\R)}:=\text{ TV}(\rho(t,\cdot);\R),&&&\\
        & (b) = \sup_{t \in [0,T]} {\left\lVert \rho(t,\cdot)\right\rVert}_{L^1(\R)}, &&&\\
        & (c) = 2\int_{0}^{T} {\left\lVert \rho(t,\cdot)-\sigma(t,\cdot)\right\rVert}_{L^1(\R)} \,dt,
    \end{align*}
    and therefore
    \begin{align*}
    {\left\lVert \rho(t,\cdot)-\sigma(t,\cdot) \right\rVert}_{L^1(\R)} \leq  {\left\lVert \rho_0-\sigma_0 \right\rVert}_{L^1(\R)}
    +{\textcolor{blue}{K_\gep}} \int_{0}^{T} {\left\lVert \rho(t,\cdot)-\sigma(t,\cdot)\right\rVert}_{L^1(\R)} \,dt,
    \end{align*}
    with
    \begin{align*}
        {\textcolor{blue}{K_\gep}}= \textcolor{blue}{\lVert v_\gep' \rVert} \Biggl( 
                 W_\eta(0)  \left(
                    \sup_{t \in [0,T]}{\left\lVert \rho(t,\cdot)\right\rVert}_{\text{BV}(\R)} 
                    +2 \sup_{t \in [0,T]}{\left\lVert \rho(t,\cdot)\right\rVert}_{L^1(\R)}                   \right)
                +\lVert W_\eta' \rVert \sup_{t \in [0,T]}{\left\lVert \rho(t,\cdot)\right\rVert}_{L^1(\R)}   \Biggr).
    \end{align*}
    By Gronwall's lemma we get the desired claim and for the choice of $\rho_0=\sigma_0$ the uniqueness of weak entropy solutions.
    \textcolor{blue}{Further, note that this bound is still a random variable as $\lVert v_\gep' \rVert=\lVert v_\gep' \rVert(\omega)$ (Def. \ref{def:norm}), which
    we will address now.}\\
    
For the second claim we state the similar bound for (\ref{eq:NV}) as derived in \cite[2.4]{Friedrich2018}:   
    \begin{align*} 
        K= \lVert v' \rVert \Biggl( 
                 W_\eta(0)  \left(
                    \sup_{t \in [0,T]}{\left\lVert \rho(t,\cdot)\right\rVert}_{\text{BV}(\R)} 
                    +2 \sup_{t \in [0,T]}{\left\lVert \rho(t,\cdot)\right\rVert}_{L^1(\R)}                   \right)
                +\lVert W_\eta' \rVert \sup_{t \in [0,T]}{\left\lVert \rho(t,\cdot)\right\rVert}_{L^1(\R)}   \Biggr).
    \end{align*}
    \textcolor{blue}{Therefore, the
      deterministic bound is the same as for (\ref{eq:sNV}) except for $\lVert v_\gep' \rVert(\omega)$. 
    As shown in (\ref{eq:Dv_e_norm_approx_final})
    it holds that $\lVert v_\gep' \rVert(\omega) \leq \lVert v' \rVert \; \forall \omega \in \Omega$. Thus,
    the actual bound $K_\gep(\omega)$ for our stochastic model might even be lower than for (\ref{eq:NV}).
    In any case we obtain general uniqueness and can deterministically bound $K_\gep(\omega)$ by $K$.}
\end{proof} 

\end{document}